\documentclass[a4paper, reqno]{amsart}
\usepackage{graphicx} % Required for inserting images
\usepackage[utf8]{inputenc}
\usepackage{amsmath}
\usepackage{amsthm}
\usepackage{amssymb}
\usepackage[foot]{amsaddr}
\usepackage[hidelinks,colorlinks=true]{hyperref}
\hypersetup{allcolors=[rgb]{0,0.31,0.62}}
\usepackage{thmtools}   % Lar deg bruke \autoref.
\usepackage{mathrsfs}
\usepackage{enumitem}
\usepackage{mathtools}
\usepackage{caption}
\usepackage{subcaption}

\usepackage{tikz}
\usetikzlibrary{matrix} % For å kunne bruke \diagram må du ha denne.
\newcommand{\diagram}[3]{\matrix (#1) [matrix of math nodes,row
  sep={#2},column sep={#3},text height=1.5ex,text
  depth=0.25ex]}
\usepackage{tikz-cd}

\theoremstyle{plain}
\newtheorem{theorem}{Theorem}[section]
\newtheorem{proposition}[theorem]{Proposition}
\newtheorem{corollary}[theorem]{Corollary}
\newtheorem{lemma}[theorem]{Lemma}

\theoremstyle{definition}
\newtheorem{definition}[theorem]{Definition}
\newtheorem{example}[theorem]{Example}
\newtheorem{remark}[theorem]{Remark}

\DeclareMathOperator{\id}{id}
\DeclareMathOperator{\colim}{colim}

\DeclareMathOperator{\op}{op}
\DeclareMathOperator{\Ima}{im}

\newcommand{\R}{\mathbb{R}} % Tavlefet R - \R

\newcommand{\Fb}{\mathbb{Fb}}

\newcommand{\A}{\mathscr{A}}
\newcommand{\B}{\mathscr{B}}
\newcommand{\C}{\mathscr{C}}

\newcommand{\Top}{\mathsf{Top}}
\newcommand{\Mod}{\mathsf{Mod}}

\newcommand{\SCplx}{\mathsf{SCplx}}

\newcommand{\Vecs}{\mathsf{Vec}}

\newcommand{\Po}{\mathcal{P}}
\newcommand{\X}{\mathcal{X}}
\newcommand{\Yc}{\mathcal{Y}}

\newcommand{\Nn}{\mathbb{N}}
\DeclareMathOperator{\coker}{coker}
\DeclareMathOperator{\Fun}{Fun}

\DeclareMathOperator{\Lan}{Lan}
\DeclareMathOperator{\Ran}{Ran}

\DeclareMathOperator{\tfib}{tfib}
\DeclareMathOperator{\tcofib}{tcofib}
\DeclareMathOperator{\fib}{fib}
\DeclareMathOperator{\cofib}{cofib}
\DeclareMathOperator{\Cr}{cr}
\DeclareMathOperator{\jdim}{jdim}
\DeclareMathOperator{\mdim}{mdim}
\DeclareMathOperator{\PC}{PC}
\DeclareMathOperator{\CC}{CC}

\numberwithin{equation}{subsection}

\newtheorem*{proposition*}{Proposition}
\newtheorem*{theorem*}{Theorem}

\newsavebox{\pullback}
\sbox\pullback{%
\begin{tikzpicture}%
\draw (0,1ex) -- (1ex,1ex);%
\draw (1ex,0ex) -- (1ex,1ex);%
\end{tikzpicture}}

\title{Cross effects for functors from posets}
\author{Bjørnar Gullikstad Hem}
\address{École Polytechnique Fédérale de Lausanne}
\email{bjornar.hem@epfl.ch}
\date{\today}

\begin{document}

\maketitle

\begin{abstract}
We establish a precise relationship between functor calculus and the projective dimension of multipersistence modules.
Specifically, we develop a new notion of functor calculus for functors from posets, which detects vanishing total fibers of cubes.
We give an explicit construction of the universal approximation functors of this functor calculus.
We then use these approximations to prove two new theorems, providing necessary and sufficient conditions for an $n$-parameter multipersistence module to have projective dimension at most $n-1$ and at most $n-2$.
\end{abstract}

\setcounter{tocdepth}{1}
\tableofcontents

\section{Introduction}

\subsection{Background}

Within the field of topological data analysis (TDA), multipersistent homology \cite{multipersistence_source_1, multipersistence_source_2} is the generalization of persistent homology to filtrations by two or more parameters.
Multipersistence arises in many different situations where the need for several parameters becomes apparent, such as time-varying data \cite{xian2020capturing, bhaskar2019analyzing} or data with outliers or variations in density \cite{vipond2021multiparameter, lesnick2015interactive}.
While 1-parameter persistence modules always decompose as a direct sum of simple components, called \emph{interval modules} \cite{interval_decomp_source, structure_theorem_crawley, structure_theorem_webb}, the same is not true for multipersistence (see, e.g., \cite[Example 8.3]{BotnanLesnickMultipersistence}).
Therefore, an important research topic in TDA is understanding the representation theory of multipersistence modules.

One of the most fundamental invariants of multipersistence modules is their projective dimension, i.e., the length of their minimal projective resolution.
It is a well known fact that multipersistence modules admit finite projective resolutions; in fact, an $n$-parameter persistence module has projective dimension at most $n$ \cite{BotnanLesnickMultipersistence}.
Moreover, the minimal projective resolution can be computed explicitly, using Koszul complexes \cite{realisationsposetstameness}.

The field of functor calculus consists of several closely related mathematical theories, all revolving around the idea of approximating a functor by a tower of other functors often referred to as \emph{degree $n$ approximations}.
Examples of variants of functor calculus include Goodwillie calculus \cite{goodwillie_calculus, goodwillie_calculus2, goodwillie_calculus3}, manifold calculus \cite{Weiss_1999, Goodwillie_1999} and abelian calculus \cite{abelian_calculus}.
A dual notion of functor calculus, called functor \emph{cocalculus}, also exists, with examples including \cite{dual_calculus}.

In \cite{hem2025posetfunctorcocalculusapplications}, we introduced a new flavor of functor cocalculus, \emph{poset cocalculus}.
Given a functor $F$ from a poset (satisfying certain conditions\footnote{In particular, the poset should be a distributive lattice.}) to a (sufficiently nice\footnote{See \cite[1.4]{hem2025posetfunctorcocalculusapplications} for the exact requirements.}) model category, poset cocalculus produces a cotower of codegree $n$ approximations, i.e., a commutative diagram,
\begin{equation*}
\begin{aligned}
\begin{tikzpicture}
\diagram{d}{3em}{3em}{
    \vdots & \ \\
    T_1 F & \ \\
    T_0 F & F, \\
};

\path[->,font = \scriptsize, midway]
(d-2-1) edge (d-1-1)
(d-1-1) edge (d-3-2)
(d-3-1) edge (d-2-1)
(d-2-1) edge (d-3-2)
(d-3-1) edge (d-3-2);
\end{tikzpicture}
\end{aligned}
\end{equation*}
These codegree $n$ approximations are defined in terms of a universal property involving preservation of certain colimits. Furthermore, taking the codegree $n$ approximation of a persistence module is stable under an appropriate notion of generalized interleaving distance.
The theory also dualizes, yielding the functor calculus \emph{poset calculus}.

In \cite{hem2025decomposingmultipersistencemodulesusing}, we applied poset (co)calculus to the representation theory of multipersistence modules.
We showed that the codegree 1 approximation captures the notion of \emph{middle exactness} \cite{botnanMiddleExactness}, a local condition for interval decomposability that has been used to show that a class of multifiltrations called \emph{interlevel persistence} \cite{interlevelset_persistence_carlsson} gives rise to interval-decomposable multipersistence modules.
We further utilized poset cocalculus to establish new necessary conditions for interval decomposability of multipersistence modules.

\subsection{Main contributions}

In this paper, we develop a new theory of functor cocalculus, called \emph{cross-poset cocalculus}.
In this cocalculus, a functor $F$ from a lattice to a pointed, cocomplete category is said to be \emph{cross-codegree} $n$ if, for every strongly bicartesian $(n+1)$-cube $\X \colon \Po([n]) \to P$, the total cofiber of $F \circ \X$ is 0 (these concepts will be explained in \autoref{sec:prelims}).
This is a strictly weaker condition than the ``codegree $n$'' notion used in poset cocalculus.

We give a construction of the universal cross-codegree $n$ approximation, $\Gamma_n$, in the case where the target category is abelian.
Remarkably, these approximations can be defined succinctly from the $T_n$ functors of poset cocalculus, in the following way:
\begin{equation*}
    \Gamma_n F = \Ima(T_n F \to F).
\end{equation*}
For every $n \ge 0$, $\Gamma_n$ admits the structure of an idempotent comonad, and these comonads fit into a cotower of the following form,
\begin{center}
    \begin{tikzcd}
    \vdots \arrow[rddd, tail]               &   \\
    \Gamma_2 F \arrow[u, tail] \arrow[rdd, tail] &   \\
    \Gamma_1 F \arrow[rd, tail] \arrow[u, tail]  &   \\
    \Gamma_0 F \arrow[r, tail] \arrow[u, tail]   & F.
    \end{tikzcd}
\end{center}
Furthermore, all the arrows in this diagram are monomorphisms.
The following results show that $\Gamma_n$ is indeed the universal cross-codegree $n$ approximation functor.
\begin{theorem*}[\autoref{thm:theoremA}]
    Let $P$ be a distributive lattice, $\A$ an abelian category and $F \colon P \to \A$ a functor. Then,
    $\Gamma_n F$ is cross-codegree $n$.
\end{theorem*}
\begin{theorem*}[\autoref{thm:theoremB}]
    Let $P$ be a join-factorization lattice, $\A$ an abelian category, and $F \colon P \to \A$ a functor.
    If $F$ is cross-codegree $n$, then $\Gamma_n F \cong F$.
\end{theorem*}
\begin{proposition*}[\autoref{prop:universal_prop_cross}]
    Let $P$ be a join-factorization lattice, $\A$ an abelian category, and $F \colon P \to \A$ a functor.
    Then, every natural transformation $\alpha \colon G \to F$ from a cross-codegree $n$ functor $G$ factors uniquely through $\Gamma_nF \to F$.
\end{proposition*}
The definition of a join-factorization lattice is given in \autoref{sec:prelims}, and includes, in particular, finite distributive lattices.

We also dualize all the definitions and results above, giving rise to the \emph{cross-poset calculus}. In particular, the universal cross-degree $n$ approximation of a functor $F$ is denoted $\Gamma^n F$.
We prove several interesting properties of the $\Gamma_n$ and $\Gamma^n$ functors, including the following.
\begin{theorem*}[\autoref{thm:dist_law}]
    Let $P$ be a distributive lattice, $\A$ an abelian category and $F \colon P \to \A$ a functor. For any $m,n \ge 0$,
    \begin{equation*}
        \Gamma_m \Gamma^n F \cong \Gamma^n \Gamma_m F.
    \end{equation*}
    Furthermore, this isomorphism is natural in $F$.
\end{theorem*}

We further explore connections of cross-poset calculus to multipersistence. We prove the following two results, showing a connection between cross-poset calculus and projective resolution.
\begin{theorem*}[\autoref{thm:pdim_thm1}]
    Let $P$ be a finite distributive lattice of dimension $n$, with $n \ge 1$, and $F \colon P \to \Vecs_{\Fb}$ a finitely generated persistence module. Then the following are equivalent.
    \begin{enumerate}
        \item $F$ has projective dimension at most $n-1$.
        \item $F$ is cross-degree $n-1$.
        \item $F \cong \Gamma^{n-1} F$.
    \end{enumerate}
\end{theorem*}
\begin{theorem*}[\autoref{thm:pdim_thm2}]
    Let $P$ be a finite distributive lattice of dimension $n$, with $n \ge 2$, and $F \colon P \to \Vecs_{\Fb}$ a finitely generated persistence module. Then the following are equivalent.
    \begin{enumerate}
        \item $F$ has projective dimension at most $n-2$.
        \item $F$ is degree $n-1$ and cross-degree $n-2$.
        \item $F \cong T^{n-1} \Gamma^{n-2} F$.
    \end{enumerate}
\end{theorem*}
Here the dimension of a poset is the minimal
number of total orders such that the poset embeds into their product.

We conclude with examples of applications where multipersistence modules with low cross-degree or cross-codegree arise. In particular, we mention an image analysis application where the bipersistence module produced is cross-degree 1, and we apply the preceding theorems to conclude that the module thus has projective dimension at most 1.

\subsection{Organization}

In section 2, we give the necessary preliminaries on poset cocalculus. In section 3, we review preliminaries on total fibers and total cofibers of cubes, and provide a formula for their computation.
In section 4, we develop the theory of \emph{cross-poset (co)calculus}, a new theory of functor (co)calculus for functors out of posets, and prove several results describing its main properties and the relation to the existing theory of poset cocalculus.
Finally, in section 5, we showcase some applications to multipersistent homology, including proving two theorems on the relation to projective dimension and applying these to concrete examples of multifiltrations.

\subsection{Notation}

We let $[n] = \{0, \dots, n\}$. The power set of a set $S$ is denoted $\Po(S)$, and we consider it as a poset ordered by inclusion.

For two elements $x$ and $y$ in a poset $P$, we let $\langle x,y \rangle = \{z \in P : x \le z \le y\}$.

For two posets $P$ and $Q$, the product partial order on $P \times Q$ is given by \begin{equation*}
(p,q) \le (p',q') \iff p \le p' \textrm{ and } q \le q'.
\end{equation*}
Unless otherwise specified, we endow products of posets with this partial order.

Given a category $\C$ and a small category $I$, we let $\Fun(I, \C)$ denote the category of functors from $I$ to $\C$.
Given categories $\A, \B$ and $\C$, and functors $\alpha \colon \A \to \B$ and $F \colon \A \to \C$, we let $\Lan_{\alpha} F \colon \B \to \C$ denote the left Kan extension of $F$ along $\alpha$, whenever it exists. Similarly, we let $\Ran_\alpha F \colon \B \to \C$ denote the right Kan extension of $F$ along $\alpha$, whenever it exists.

\subsection*{Acknowledgments}

I am grateful to my supervisor, Kathryn Hess, for her guidance, insightful feedback, and many fruitful discussions.
I also thank my colleagues in the topology group at EPFL, for creating a stimulating and supportive research environment.

\section{Preliminaries on poset cocalculus}

\label{sec:prelims}

We present the theory of poset cocalculus, as introduced in \cite{hem2025posetfunctorcocalculusapplications}.
For further examples, especially of poset cocalculus applied to persistence modules, we refer to \cite{hem2025decomposingmultipersistencemodulesusing}.

\subsection{Posets and lattices}

Given two elements $x, y$ of a poset, we denote their \emph{least upper bound}, or \emph{join}, by $x \vee y$ and their \emph{greatest lower bound}, or \emph{meet}, by $x \wedge y$ (whenever they exist). A \emph{lattice} is a poset in which $x \vee y$ and $x \wedge y$ exist for all pairs of elements $(x, y)$.

An element $x$ in a poset $P$ is \emph{minimal} if there is no $y \in P$ such that $y < x$.
We similarly say that an element $x$ in a poset $P$ is \emph{maximal} if there is no $y \in P$ such that $y > x$.

\begin{definition}\label{def:descending_chain}
A poset $P$ satisfies the \emph{descending chain condition} if every nonempty subset of $P$ has a minimal element.

\end{definition}

Finite posets trivially satisfy the descending chain condition.

\subsection{Distributive lattices}

A lattice $P$ is \emph{distributive} if for all elements $x,y,z \in P$,
\begin{equation}\label{eq:dist_lattice}
x \wedge (y \vee z) = (x \wedge y) \vee (x \wedge z).
\end{equation}

\begin{definition}
We say that an element $v$ in a lattice $P$ is \emph{join-irreducible} if $v \neq x \vee y$ for all $x,y < v$.
\end{definition}
\begin{definition}
Let $P$ be a lattice, and let $v \in P$. A \emph{join-decomposition} of $v$ (of size $k$) is a finite collection of elements $p^0, \dots, p^{k-1} \in P$ such that $v = p^0 \vee \dots \vee p^{k-1}$.

\end{definition}
\begin{definition}\label{def:dim}
Let $P$ be a distributive lattice. For an element $v \in P$, we define the \emph{join-dimension} of $v$, denoted $\jdim(v)$, as the minimal $k$ such that $v$ has a size $k$ join-decomposition into join-irreducible elements (i.e., if $v$ can be written as the join of $k$ join-irreducible elements).
\end{definition}
We adopt the convention that if $v$ is minimal, then $\jdim(v) = 0$. Furthermore, if $v$ has no join-decomposition into join-irreducible elements, then $\jdim(v) = \infty$.
For further details on the join-dimension, including equivalent formulations, we refer to \cite[Section 4.2]{hem2025posetfunctorcocalculusapplications}.

The following proposition says that in distributive lattices satisfying the descending chain condition, every element has finite join-dimension.
\begin{proposition}{\cite[Theorem 9, p.\ 142]{Birkhoff}}\label{thm:Birkhoff}
If $P$ is a distributive lattice that satisfies the descending chain condition, then every element of $P$ has a unique minimal join-decomposition into join-irreducible elements.
\end{proposition}

Given a distributive lattice $P$, we let $P_{\le n}$ denote the subposet of elements with join-dimension $\le n$. 

All notions above admit dual versions obtained by exchanging joins and meets.
In particular, we define the meet-dimension $\mdim(v)$ and the subposet $P^{\le n}$ analogously.

\subsection{Cubical diagrams}

Let $[k] = \{0, \dots, k\}$, and let $\Po([k])$ be the power set of $[k]$, viewed as a poset ordered by inclusion. 
For $\C$ any category, we refer to functors $\X \colon \Po([k-1]) \to \C$ as \emph{$k$-cubes} in $\C$. 

\begin{definition}
Let $\C$ be a category.
A $(k+1)$-cube $\X \colon \Po([k]) \to \C$ is \emph{cocartesian} if the canonical map
\begin{equation*}
\underset{S \subsetneq [k]}{\colim \X(S)} \to \X([k])
\end{equation*}
is an isomorphism.

Similarly, $\X$ is \emph{cartesian} if the canonical map
\begin{equation*}
\X(\emptyset) \to \underset{S \subseteq [k], S \neq \emptyset}{\lim \X(S)}
\end{equation*}
is an isomorphism.
\end{definition}

\begin{definition}
Let $\C$ be a category.
A $(k+1)$-cube $\X \colon \Po([k]) \to \C$ is \emph{strongly cartesian} (resp. \emph{strongly cocartesian}) if each face of dimension at least 2 is cartesian (resp. cocartesian).

If $\X$ is both strongly cartesian and strongly cocartesian, it is called \emph{strongly bicartesian}.
\end{definition}

We now characterize strongly bicartesian cubes in a distributive lattice. Note that for any finite diagram $\alpha \colon I \to P$ into a lattice, $\colim_I \alpha$ is equal to the join $\bigvee_{i \in I}\alpha(i)$, and dually, $\lim_I {\alpha} = \bigwedge_{i \in I}\alpha(i)$.
\begin{definition}
Let $P$ be a lattice, let $v \in P$, and let $k$ be a positive integer.
A \emph{pairwise cover} of $v$ of size $k$ is a collection of elements $x^0, \dots, x^{k-1} \in P$, with $x^i \le v$ for all $i$, such that $x^i \vee x^j = v$ for all $i \neq j$.
\end{definition}
\begin{definition}\label{def:poset_cube}
Let $P$ be a lattice, and let $v \in P$. Let further $x^0, \dots, x^k$ be a pairwise cover of $v$.
We define the $(k+1)$-cube 
\begin{equation*}
\X_{x^0, \dots, x^k} \colon \Po([k]) \to P
\end{equation*}
as follows.
\begin{equation*}
\X_{x^0, \dots, x^k} (S) = 
\begin{cases}
v, &\quad S = [k], \\ 
\bigwedge_{i \notin S} x^i, &\quad \text{ otherwise.} \\
\end{cases}
\end{equation*}
\end{definition}
The following are Lemma 3.6 and Lemma 3.7 in \cite{hem2025posetfunctorcocalculusapplications}.
\begin{lemma}\label{lem:bicartesian_from_codecomp}
Let $P$ be a distributive lattice. For every $v \in P$ and every pairwise cover $x^0, \dots, x^k$ of $v$, the cube $\X_{x^0, \dots, x^k}$ is strongly bicartesian.
\end{lemma}
\begin{lemma}\label{lem:codecomp_from_bicartesian}
Let $P$ be a lattice, and let $\X \colon \Po([k]) \to P$ be a strongly bicartesian cube. For $i \in [k]$, %\footnote{Recall that $[k] = \{0, \dots, k\}.$} 
let $x^i = \X\big([k] \setminus \{i\}\big)$. Then $x^0, \dots, x^k$ is a pairwise cover of $\X\big([k]\big)$.

Furthermore, $\X = \X_{x^0, \dots, x^k}$, as defined in \autoref{def:poset_cube}.
\end{lemma}

\begin{remark}\label{rmk:cubes_bicart_meet}
    For any cube into any category, if all faces of dimension 2 are (co)cartesian, then so are all faces of dimension $> 2$ \cite[p.\ 272]{cubicalhtpy}. Hence, a cube is strongly (co)cartesian if all faces of dimension 2 are (co)cartesian.

    In particular, when $P$ is a lattice, an $n$-cube $\X \colon \Po([n-1]) \to P$ is strongly cocartesian if and only if
    \begin{equation*}
        \X(S \cup T) = \X(S) \vee \X(T)
    \end{equation*}
    for all $S,T \subseteq [n-1]$. 
    Likewise, it's strongly cartesian if and only if
    \begin{equation*}
        \X(S \cap T) = \X(S) \wedge \X(T)
    \end{equation*}
    for all $S,T \subseteq [n-1]$.
    In other words, an $n$-cube is strongly cocartesian if and only if it preserves join, strongly cartesian if and only if it preserves meet, and strongly bicartesian if and only if it preserves both join and meet.
    In the case where $\X \colon \Po([n-1]) \to P$ is injective, it is strongly bicartesian if and only if it is a sublattice.
\end{remark}

\subsection{Poset (co)calculus}

We recall here the main definitions and results of the \emph{poset cocalculus} introduced in \cite{hem2025posetfunctorcocalculusapplications}.

\begin{definition}
Let $P$ be a distributive lattice, and $\C$ a finitely cocomplete category. A functor $F \colon P \to \C$ is \emph{codegree $n$} if it takes strongly bicartesian $(n+1)$-cubes to cocartesian $(n+1)$-cubes.
\end{definition}

Given a functor $F \colon P \to \C$, we denote by $T_n F$ its \emph{codegree $n$ approximation} (defined as in \cite[Definition 4.30]{hem2025posetfunctorcocalculusapplications}).
That is, $T_n F$ is the left Kan extension of the restriction $F|_{P_{\le n}}$ along the inclusion $i \colon P_{\le n} \hookrightarrow P$.
This can be computed objectwise as
\begin{equation}\label{eq:T_n_cat}
T_n F (x) = \underset{v \in P_{\le n}, v \le x}{\colim} F(v).
\end{equation}
The universal property of left Kan extensions gives a canonical natural transformation, $\varepsilon_n \colon T_n F \to F$, natural in $F$.
Observe that $T_n$ commutes with colimits in $\Fun(P, \C)$, as colimits commute with colimits.

The following two theorems are Theorem 4.31 and Theorem 4.32 in \cite{hem2025posetfunctorcocalculusapplications}.
\begin{theorem}\label{theoremA_cat}
If $P$ is a distributive lattice, then for every functor $F \colon P \to \C$ to a cocomplete category, $T_k F$ is codegree $k$.
\end{theorem}

Let $P$ be a distributive lattice with a minimal element. We say that $P$ is a \emph{join-factorization lattice} if there exists a distributive lattice $Q$, satisfying the descending chain condition, and an order-preserving function $f \colon P \to Q$ such that for all $v \in P$,
\begin{equation*}
\jdim(f(v)) = \jdim(v).
\end{equation*}

\begin{theorem}\label{theoremB_cat}
Let $P$ be a join-factorization lattice, and let $F \colon P \to \C$ be a functor to a cocomplete category. If $F$ is codegree $n$, then $\varepsilon_n \colon T_n F \to F$ is a natural isomorphism.
\end{theorem}

Examples of join-factorization lattices include distributive lattices satisfying the descending chain condition (and, in particular, finite distributive lattices) \cite[Example 4.17]{hem2025posetfunctorcocalculusapplications}.
Furthermore, if $T_1, \dots, T_n$ are total orders with minimal elements, then $P = T_1 \times \dots \times T_n$ is a join-factorization lattice \cite[Example 4.18]{hem2025posetfunctorcocalculusapplications}.

Codegree $n$ approximation satisfies the following universal property.
\begin{proposition}
    Let $P$ be a join-factorization lattice, and let $F \colon P \to \C$ be a functor to a cocomplete category.
    Then, every natural transformation $\alpha \colon G \to F$ from a codegree $n$ functor $G$ factors uniquely through $T_nF \to F$.
\end{proposition}

All constructions above admit dual versions, yielding \emph{poset calculus}.

\begin{itemize}
    \item A functor $F \colon P \to \C$ is \emph{degree $n$} if it takes strongly bicartesian $(n+1)$-cubes to cartesian $(n+1)$-cubes.
    \item  Given a functor $F \colon P \to \C$, we denote by $T^n F$ its \emph{degree $n$ approximation}, defined as 
    the right Kan extension of the restriction $F|_{P^{\le n}}$ along $i \colon P^{\le n} \hookrightarrow P$. We let $\eta_n \colon F \to T^n F$ be the canonical natural transformation.
    \item Dual versions of \autoref{theoremA_cat} and \autoref{theoremB_cat} hold for poset calculus \cite[Section 8]{hem2025posetfunctorcocalculusapplications}.
\end{itemize}

\section{Total (co)fibers}

We review the theory of total fibers and total cofibers of cubes, and describe how to compute them. Note that all the notions here are 1-categorical, that is, the limits are \emph{not} homotopy limits.

\subsection{Total fibers of cubes}

\begin{definition}\label{def:tofib}
Let $\C$ be a pointed, finitely complete category and $\X \colon \Po([n-1]) \to \C$ an $n$--cube in $\C$.
The \emph{total fiber} of $\X$, denoted $\tfib \X$, is defined inductively as
\begin{equation*}
  \tfib (\X) = \begin{cases}
    \fib\left(\X(\emptyset) \to \X(\{0\})\right), &\quad n = 1, \\
    \fib \left( \tfib (\X|_{\langle \emptyset, [n-2] \rangle}) \to \tfib (\X|_{\langle \{n-1\}, [n-1]\rangle}) \right), &\quad n > 1.
  \end{cases}
\end{equation*}
\end{definition}

\begin{lemma}\label{lemma:tofib}
  Let $\C$ be a pointed, finitely complete category and $\X \colon \Po([n-1]) \to \C$ an $n$--cube in $\C$.
  Then
  \begin{equation*}
    \tfib \X \cong \fib \left( \X(\emptyset) \to \underset{S \subseteq [n-1], S \neq \emptyset}{\lim} \X(S)\right).
  \end{equation*}
\end{lemma}

\begin{proof}
  This is a standard result for cubical diagrams.
  A proof is given in \cite[Proposition 5.5.4]{cubicalhtpy} for homotopy limits in topological spaces. This proof is straightforward and easily adapted to the case of limits in a pointed category.
\end{proof}

\begin{lemma}\label{lemma:tofib_prod}
  Let $\C$ be a pointed, finitely complete category and $\X \colon \Po([n-1]) \to \C$ an $n$--cube in $\C$.
  Then
  \begin{equation*}
    \tfib \X \cong \fib \left( \X(\emptyset) \to \prod_{S \subseteq [n-1], |S| = 1} \X(S) \right).
  \end{equation*}
\end{lemma}

\begin{proof}
  Let $\gamma$ denote the canonical map $\gamma \colon \X(\emptyset) \to \lim_{S \subseteq [n-1], S \neq \emptyset} \X(S)$, and let $\gamma'$ denote the canonical map $\X(\emptyset) \to \prod_{S \subseteq [n-1], |S|=1} \X(S)$.
  By \autoref{lemma:tofib}, it suffices to show that $\fib \gamma \cong \fib \gamma'$.
  We prove this by checking the universal properties.

  First, let $f \colon A \to \X(\emptyset)$ be such that $\gamma \circ f = 0$. Then for any $S \subseteq [n-1]$ with $|S| = 1$, $\X(\emptyset \subseteq S) \circ f$ equals the composition
  \begin{equation*}
      A \xrightarrow{f} \X(\emptyset) \xrightarrow{\gamma} \lim_{T \subseteq [n-1], T \neq \emptyset} \X(T) \to \X(S),
  \end{equation*}
  which equals $0$. Hence, $\gamma' \circ f = 0$.

  Now, let $f \colon A \to \X(\emptyset)$ be such that $\gamma' \circ f = 0$.
  Let $S \subseteq [n-1]$ with $S \neq \emptyset$. Let $i \in S$. Then,
  \begin{equation*}
    \X(\emptyset \subseteq S) \circ f = \X(\{i\} \subseteq S) \circ \X(\emptyset \subseteq \{i\}) \circ f,
  \end{equation*}
  which equals the composition
  \begin{equation*}
    A \xrightarrow{f} \X(\emptyset) \xrightarrow{\gamma'} \prod_{T \subseteq [n-1], |T| = 1} \X(T) \to \X({i}) \to \X(S),
  \end{equation*}
  which equals 0. Hence, $\gamma \circ f = 0$.
\end{proof}

\subsection{Total cofibers of cubes}

\begin{definition}\label{def:tocofib}
Let $\C$ be a pointed, finitely cocomplete category and $\X \colon \Po([n-1]) \to \C$ an $n$--cube in $\C$.
The \emph{total cofiber} of $\X$, denoted $\tcofib \X$, is defined inductively as
\begin{equation*}
  \tcofib (\X) = \begin{cases}
    \cofib\left(\X(\emptyset) \to \X(\{0\})\right), \quad n = 1, \\
    \cofib \left( \tcofib (\X|_{\langle \emptyset, [n-2] \rangle}) \to \tcofib (\X|_{\langle \{n-1\}, [n-1]\rangle}) \right), \quad n > 1.
  \end{cases}
\end{equation*}
\end{definition}

\begin{lemma}\label{lemma:tocofib}
  Let $\C$ be a pointed, finitely cocomplete category and $\X \colon \Po([n-1]) \to \C$ an $n$--cube in $\C$.
  Then
  \begin{equation*}
    \tcofib \X \cong \cofib \left( \underset{S \subsetneq [n-1]}{\colim} \X(S) \to \X([n-1]) \right).
  \end{equation*}
\end{lemma}

\begin{proof}
  The proof is dual to \autoref{lemma:tofib}.
  
\end{proof}

\begin{lemma}\label{lemma:tocofib_coprod}
  Let $\C$ be a pointed, finitely cocomplete category and $\X \colon \Po([n-1]) \to \C$ an $n$--cube in $\C$.
  Then
  \begin{equation*}
    \tcofib \X \cong \cofib \left( \coprod_{S \subseteq [n-1], |S| = n-1} \X(S) \to \X([n-1]) \right).
  \end{equation*}
\end{lemma}

\begin{proof}
  The proof is dual to \autoref{lemma:tofib_prod}.
\end{proof}

\section{Cross-poset cocalculus}

In this section, we develop a new variant of functor (co)calculus for functors from a poset to a pointed category, based on the theory of cross effects. We develop the dual notions of cross-degree and cross-codegree, and give a construction for the universal cross-(co)degree $n$ approximations.
We further explore some properties of these approximation functors.

\subsection{Cross-(co)degree}

We define the notions of \emph{cross-degree} and \emph{cross-codegree}. The terminology comes from the related notion of \emph{cross effects} from abelian calculus \cite{abelian_calculus}, which describe vanishing total fibers.

\begin{definition}\label{def:crossdeg}
    Let $\C$ be a pointed, finitely complete category, and $P$ a lattice. 
    Let $F \colon P \to \C$ be a functor.
    We say that $F$ is \emph{cross-degree $n$} if, for every strongly bicartesian $(n+1)$--cube $\X \colon \Po([n]) \to P$,
    \begin{equation*}
        \tfib (F \circ \X) \cong 0.
    \end{equation*}
\end{definition}

\begin{definition}\label{def:crosscodeg}
    Let $\C$ be a pointed, finitely cocomplete category, and $P$ a lattice. 
    Let $F \colon P \to \C$ be a functor.
    We say that $F$ is \emph{cross-codegree $n$} if, for every strongly bicartesian $(n+1)$--cube $\X \colon \Po([n]) \to P$,
    \begin{equation*}
        \tcofib (F \circ \X) \cong 0.
    \end{equation*}
\end{definition}

\begin{remark}
    If $F$ is cross-degree $n$, then $F$ is also cross-degree $n+1$.
    To see this, observe that if $\X \colon \Po([n+1]) \to P$ is a strongly bicartesian $(n+2)$--cube, then 
    \begin{align*}
        \tfib(F \circ \X) &\cong \fib \left( \tfib (F \circ \X|_{\langle \emptyset, [n] \rangle}) \to \tfib (F \circ \X|_{\langle \{n+1\}, [n+1]\rangle}) \right) \\
        &\cong \fib(0 \to 0) \cong 0.
    \end{align*}
    Dually, if $F$ is cross-codegree $n$, then $F$ is cross-codegree $n+1$.
\end{remark}

\begin{remark}
    It follows from \autoref{lemma:tofib} that if $F \colon P \to \C$ is degree $n$, then it is cross-degree $n$. Indeed, a functor is degree $n$ if the comparison map 
    \begin{equation*}
      \X(\emptyset) \to \underset{S \subseteq [n-1], S \neq \emptyset}{\lim} \X(S)
    \end{equation*}
    is an isomorphism, which is a strictly stronger condition than its fiber being 0 (in particular, the fiber of any monomorphism is 0).
    Likewise, if $F$ is codegree $n$, then it is cross-codegree $n$.
\end{remark}

\begin{remark}
    One could generalize \autoref{def:crossdeg} (and \autoref{def:crosscodeg}) to the context of functors from a lattice to a model category, by replacing total (co)fibers with total homotopy (co)fibers.
    We remark that, in the case where the target category is a stable model category \cite[Section 7]{hovey}, 
    a functor would then be degree $n$ if and only if it is cross-degree $n$, as in a stable model category, a cube is homotopy cartesian if and only if its total homotopy fiber is 0.
    However, we will restrict ourselves to the context of ordinary limits and colimits in this text, in particular because the construction of the approximation functors presented in \autoref{def:Pn} does not generalize easily to the model category setting.
\end{remark}

\subsection{(co)cross effects}

We introduce the notion of \emph{cross effects}, and the dual notion \emph{cocross effects}, which measures how far a functor is from being cross-(co)degree $n$. This leads to an equivalent condition for a functor being cross-(co)degree $n$.
We restrict to the case where the source poset is a distributive lattice, as this is the setting in which the main theorems of poset (co)calculus, as described in \autoref{sec:prelims}, apply.

\begin{definition}\label{def:cocross_eff}
    Let $\C$ be a pointed, finitely cocomplete category, and $P$ a distributive lattice. 
    Let $F \colon P \to \C$ be a functor.
    The \emph{$n$th cocross effect} of $F$, $\Cr_n F \colon P \to \C$, is defined as
    \begin{equation*}
        \Cr_n F = \cofib ( T_n F \xrightarrow{\varepsilon_n} F). 
    \end{equation*}
\end{definition}

\begin{definition}\label{def:cross_eff}
    Let $\C$ be a pointed, finitely complete category, and $P$ a distributive lattice. 
    Let $F \colon P \to \C$ be a functor.
    The \emph{$n$th cross effect} of $F$, $\Cr^n F \colon P \to \C$, is defined as
    \begin{equation*}
        \Cr^n F = \fib ( F \xrightarrow{\eta_n} T^n F). 
    \end{equation*}
\end{definition}

In functor calculus, particularly abelian functor calculus, the term ``cross effect'' is used to describe a quantity that measures how far a functor is from being degree $n$ \cite{abelian_calculus, abelian_calculus_unbased}.
The following two lemmas thus motivate the terminology in \autoref{def:cocross_eff} and \autoref{def:cross_eff}.

\begin{lemma}\label{lemma:cross_effect_A}
    Let $\C$ be a pointed, finitely cocomplete category, and $P$ a distributive lattice. 
    Let $F \colon P \to \C$ be a functor.
    If $\Cr_n F \cong 0$, then $F$ is cross-codegree $n$.

    Dually,
    let $\C$ be a pointed, finitely complete category, and $P$ a distributive lattice. 
    Let $F \colon P \to \C$ be a functor.
    If $\Cr^n F \cong 0$, then $F$ is cross-degree $n$.
\end{lemma}

\begin{proof}
    We prove the first part. The second is dual.

    Suppose that $F \colon P \to \C$ is any functor such that $\Cr_n F$ is 0.
    Let $\X \colon \Po([n]) \to P$ be a strongly bicartesian $(n+1)$-cube. Note that as $T_n F$ is cross-codegree $n$, $\tcofib (T_n F \circ \X) \cong 0$.
    We compute the total cofiber of the $(n+2)$--cube $(T_n F \circ \X \to F \circ \X)$ in two ways.
    Firstly,
    \begin{align*}
        \tcofib (T_n F \circ \X &\to F \circ \X) \cong
        \cofib \left( \tcofib(T_n F \circ \X) \to \tcofib(F \circ \X) \right) \\
        &\cong \cofib \left( 0 \to \tcofib(F \circ \X) \right)
        \cong \tcofib(F \circ \X).
    \end{align*}
    Secondly,
    \begin{align*}
        \tcofib (T_n F \circ \X &\to F \circ \X) \cong
        \tcofib(\cofib(T_n F \circ \X \to F \circ \X)) \\
        &\cong \tcofib (\Cr_n F \circ \X) \cong 0.
    \end{align*}
    Thus, $\tcofib (F \circ \X) \cong 0$. This concludes the proof.
\end{proof}

\begin{lemma}\label{lemma:cross_effect_B}
    Let $\C$ be a pointed, finitely cocomplete category, and $P$ a join-factorization lattice. 
    Let $F \colon P \to \C$ be a functor. If $F$ is cross-codegree $n$, then $\Cr_n F \cong 0$.

    Dually,
    let $\C$ be a pointed, finitely complete category, and $P$ a meet-factorization lattice. 
    Let $F \colon P \to \C$ be a functor. If $F$ is cross-degree $n$, then $\Cr^n F \cong 0$.
\end{lemma}

\begin{proof}
    We prove the first part. The second is dual.
    
    Suppose that $F \colon P \to \C$ is cross-codegree $n$.
    We wish to show that $\Cr_nF(x) \cong 0$ for all $x \in P$.
    Note that as $P$ is a join-factorization lattice, every element has finite join-dimension.
    We proceed by induction on join-dimension.

    First, observe that if $x \in P$ with $\jdim(x) \le n$, then
    \begin{equation*}
        \Cr_nF(x) \cong \cofib(T_n F(x) \to F(x)) \cong \cofib(F(x) \to F(x)) \cong 0.
    \end{equation*}

    Now, suppose that $x \in P$ with $\jdim(x) = m > n$, and suppose that $\Cr_nF(y) \cong 0$ for all $y \in P$ with $\jdim(y) < m$.
    Let $x = p_0 \vee \dots \vee p_{m-1}$ be the unique minimal join-decomposition into join-irreducibles of $x$.
    Then, the set $\{y_0, \dots, y_{m-1}\}$ defined by
    \begin{equation*}
        y_i = \bigvee_{0 \le j \le m-1, j \neq i} p_j,
    \end{equation*}
    is a pairwise cover of $x$, and each $y_i$ has join-dimension $m-1$. Furthermore, the subset $\{y_0, \dots, y_n\}$ is also a pairwise cover of $x$.
    We compute the total cofiber of the $(n+2)$-cube
    \begin{equation*}
        \Yc = \big((T_n F \circ \X_{y_0, \dots, y_n}) \to (F \circ \X_{y_0, \dots, y_n})\big)
    \end{equation*}
    in two different ways.

    Firstly,
    \begin{align*}
        \tcofib \Yc &\cong \cofib\big(\tcofib (T_n F \circ \X_{y_0, \dots, y_n}) \to \tcofib (F \circ \X_{y_0, \dots, y_n}) \big) \\
        &\cong \cofib(0 \to 0) \cong 0,
    \end{align*}
    where we use that $F$ is cross-codegree $n$, and that $T_n F$ is codegree $n$, and thus also cross-codegree $n$.

    Secondly,
    \begin{align*}
        \tcofib \Yc &\cong \tcofib \big( \cofib\big((T_n F \circ \X_{y_0, \dots, y_n}) \to (F \circ \X_{y_0, \dots, y_n}) \big) \big) \\
        &\cong \tcofib (\Cr_{n} F \circ \X_{y_0, \dots, y_n})
        \cong \cofib\big(\coprod_{0 \le i \le n} \Cr_n F(y_i) \to \Cr_n F(x)\big) \\
        &\cong \cofib \big(0 \to \Cr_n F(x)\big) \cong \Cr_n F(x).
    \end{align*}
    This shows that $\Cr_n F(x) \cong 0$, as desired.
    This concludes the induction step.

\end{proof}

\begin{remark}
    Note that being a join-factorization lattice is a strictly stronger condition than being a distributive lattice.
    The fact that the stronger condition of being a join-factorization lattice is required for \autoref{lemma:cross_effect_B} is analogous to it being required for \cite[Theorem B]{hem2025posetfunctorcocalculusapplications}, while \cite[Theorem A]{hem2025posetfunctorcocalculusapplications} holds for all distributive lattices.
\end{remark}

\begin{lemma}\label{lemma:Tn_preserve_cross_degree}
    Let $\C$ be a pointed, finitely cocomplete category, and $P$ a join-factorization lattice. 
    If $F$ is cross-codegree $n$, then $T_k F$ is cross-codegree $n$ for all $k \ge 0$.
    
    Dually, 
    let $\C$ be a pointed, finitely complete category, and $P$ a meet-factorization lattice. 
    If $F$ is cross-degree $n$, then $T^k F$ is cross-degree $n$ for all $k \ge 0$.
\end{lemma}

\begin{proof}
    We prove the first part. The second is dual.

    We use \autoref{lemma:cross_effect_A} and \autoref{lemma:cross_effect_B}. 
    First, suppose that $k \le n$. Then,
    \begin{align*}
        \Cr_n (T_k F) &\cong \cofib (T_n T_k F \to T_k F) \cong \cofib(T_k F \to T_k F) \cong 0.
    \end{align*}
    Next, suppose that $k > n$. Then,
    \begin{align*}
        \Cr_n (T_k F) &\cong \cofib (T_n T_k F \to T_k F) \cong \cofib(T_n F \to T_k F) \\
        &\cong \cofib (T_k T_n F \to T_k F).
    \end{align*}
    As $T_k$ commutes with colimits, this is isomorphic to
    \begin{align*}
        T_k \cofib (T_n F \to F)
        \cong T_k \Cr_n F \cong T_k 0 \cong 0.
    \end{align*}
\end{proof}

\subsection{The Taylor (co)tower of cross-(co)degree $n$ approximations}

We now restrict our attention to abelian categories.
We will derive a formula for the universal cross-codegree $n$ approximation of a functor from a join-factorization lattice to an abelian category (as well as the dual notion for cross-degree).

We let $\ker f$ and $\coker f$ denote the kernel and cokernel, respectively, of a morphism $f$. Note that in an abelian category, a kernel is the same as a fiber, and a cokernel is the same as a cofiber.

\begin{definition}\label{def:Pn}
    Let $P$ be a distributive lattice, $\A$ an abelian category and $F \colon P \to \A$ a functor.

    The \emph{cross-codegree $n$ approximation}, $\Gamma_n F \colon P \to \A$, is defined as
    \begin{equation*}
        \Gamma_n F = \Ima(T_n F \to F).
    \end{equation*}

    The \emph{cross-degree $n$ approximation}, $\Gamma^n F \colon P \to \A$, is defined as
    \begin{equation*}
        \Gamma^n F = \Ima(F \to T^n F).
    \end{equation*}
\end{definition}

The cross-codegree $n$ approximations assemble into a cotower, illustrated in the following commutative diagram.
\begin{center}
    \begin{tikzcd}
    \vdots \arrow[r, two heads]          & \vdots \arrow[rddd, tail]               &   \\
    T_2 F \arrow[u] \arrow[r, two heads] & \Gamma_2 F \arrow[u, tail] \arrow[rdd, tail] &   \\
    T_1 F \arrow[u] \arrow[r, two heads] & \Gamma_1 F \arrow[rd, tail] \arrow[u, tail]  &   \\
    T_0 F \arrow[r, two heads] \arrow[u] & \Gamma_0 F \arrow[r, tail] \arrow[u, tail]   & F
    \end{tikzcd}
\end{center}
From the definition of $\Gamma_n F$, it follows that the natural transformation $T_n F \to \Gamma_n F$ is an epimorphism, and that the natural transformation $\Gamma_n F \to F$ is a monomorphism. Furthermore, for every $m > n$, the natural transformation $\Gamma_n F \to \Gamma_m F$ is a monomorphism, as the natural transformation $\Gamma_n F \to \Gamma_m F \to F$ is a monomorphism.

Dually, the cross-degree $n$ approximations assemble into a tower. The natural transformation $\Gamma^n F \to T^n F$ is a monomorphism, and the natural transformations $F \to \Gamma^n F$ and $\Gamma^m F \to \Gamma^n F$ are epimorphisms.

\begin{lemma}
    Let $P$ be a distributive lattice and $\A$ an abelian category. Then,
    \begin{itemize}
        \item $\Gamma_n$ defines an endofunctor $\Gamma_n \colon \Fun(P, \A) \to \Fun(P, \A)$.
        \item $\Gamma^n$ defines an endofunctor $\Gamma^n \colon \Fun(P, \A) \to \Fun(P, \A)$.
    \end{itemize}
\end{lemma}

\begin{proof}
We prove the first part. The second is dual.

Let $i$ denote the inclusion $i \colon P_{\le n} \hookrightarrow P$. As $T_n = \Lan_i \circ i^*$, $T_n$ is an endofunctor on $\Fun(P, \A)$. Thus, $\Gamma_n = \Ima(T_n \to \id_{\Fun(P, \A)})$ is also an endofunctor on $\Fun(P, \A)$.
\end{proof}

Not only are $\Gamma_n$ and $\Gamma^n$ endofunctors, but they are a comonad and a monad respectively.

\begin{definition}
    An \emph{idempotent monad} is a monad $(T, \mu, \eta)$ satisfying one of the following equivalent conditions \cite[Proposition 4.2.3]{borceux1994handbook}.
    \begin{enumerate}
        \item $\mu \colon TT \to T$ is an isomorphism.
        \item The maps $T(\eta), \eta_T \colon T \to TT$ are equal.
    \end{enumerate}
    
    An \emph{idempotent comonad} is the dual concept.
\end{definition}

\begin{proposition}\label{prop:idempotent_monad}
    Let $P$ be a distributive lattice and $\A$ an abelian category. Then,
    \begin{itemize}
        \item The endofunctor $\Gamma_n \colon \Fun(P, \A) \to \Fun(P, \A)$ admits the structure of an idempotent comonad.
        \item The endofunctor $\Gamma^n \colon \Fun(P, \A) \to \Fun(P, \A)$ admits the structure of an idempotent monad.
    \end{itemize}
\end{proposition}

\begin{proof}
We prove the second part. The first is dual.

Observe first that $T^n$ is an idempotent monad. It's a monad because $T^n = \Ran_i \circ i^*$, where $i$ is the inclusion $i \colon P^{\le n} \hookrightarrow P$. In other words, $T^n$ is the monad induced from the adjoint pair $(i^*, \Ran_i)$.
Furthermore, as $i$ is a fully faithful functor, the counit $i^* \circ \Ran_i \to \id$ is an isomorphism \cite[Corollary 6.3.9]{riehl_category_theory_in_context}.
Thus, $T^n$ is an idempotent monad, as the multiplication $T^n \circ T^n \to T^n$ is the natural isomorphism
\begin{equation*}
    \Ran_i \circ i^* \circ \Ran_i \circ i^* = \Ran_i \circ (i^* \circ \Ran_i) \circ i^* \xrightarrow{\cong}  \Ran_i \circ (\id) \circ i^* = T^n.
\end{equation*}

Let $\eta$ denote the unit $\eta \colon \id_{\Fun(P, \A)} \to T^n$.
Since $\Gamma^n = \Ima(\id_{\Fun(P, \A)} \to T^n)$, there are induced natural transformations
\begin{equation*}
    \bar \eta\colon \id_{\Fun(P, \A)} \to \Gamma^n,
\end{equation*}
and
\begin{equation*}
    \iota \colon \Gamma^n \to T^n.
\end{equation*}
Furthermore, $\bar \eta$ is an objectwise epimorphism, and $\iota$ is an objectwise monomorphism.

Our strategy is as follows. We first show that $\bar \eta_{\Gamma^n} \colon \Gamma^n \to \Gamma^n \Gamma^n$ is an isomorphism, and that $\Gamma^n \bar \eta = \bar \eta_{\Gamma^n}$.
We then choose $(\bar \eta_{\Gamma^n})^{-1}$ to be our monad multiplication, and show that the triple $(\Gamma^n, (\bar \eta_{\Gamma^n})^{-1}, \bar \eta)$ defines an idempotent monad.

To show that, $\bar \eta_{\Gamma^n}$ is an isomorphism, it suffices to show that for all $F \colon P \to \A$, the map $\eta_{\Gamma^n F} \colon \Gamma^n F \to T^n \Gamma^n F$ is a monomorphism (as $\Gamma^n \Gamma^n F$ is the image of this map).
Consider the commutative diagram,
\begin{center}
    \begin{tikzcd}[column sep=4em, row sep=3em]
    \Gamma^n F \arrow[r, tail, "\iota_F"] \arrow[d, "\eta_{\Gamma^nF}"] & T^n F \arrow[d, "\cong"', "\eta_{T^n F}"] \\
    T^n \Gamma^n F \arrow[r, "T^n(\iota_F)"]        & T^n T^n F.
    \end{tikzcd} 
\end{center}
As $T^n$ is an idempotent monad, $\eta_{T^n F}$ is an isomorphism.
Thus, as $\iota_F$ is a monomorphism, it follows that $\eta_{\Gamma^n F}$ is a monomorphism, as desired.

We now show that $\Gamma^n \bar \eta = \bar\eta_{\Gamma^n}$. Let $F \colon P \to \A$.
Using naturality of $\bar \eta$, we get a commutative diagram,
\begin{center}
\begin{tikzcd}[column sep=4em, row sep=3em]
F \arrow[r, "\bar \eta_F"] \arrow[d, "\bar \eta_{F}"]             & \Gamma^n F \arrow[d, "\bar \eta_{\Gamma^n F}"]               \\
\Gamma^n F  \arrow[r, "\Gamma^n \bar \eta_F"] & \Gamma^n \Gamma^n F.
\end{tikzcd}
\end{center}
Now, as $\bar \eta_F$ is an epimorphism, $\Gamma^n \bar \eta_F = \bar \eta_{\Gamma^n F}$.

Finally, we show that $(\Gamma^n, (\bar \eta_{\Gamma^n})^{-1}, \bar \eta)$ is indeed a monad.
We need to check that the following two diagrams commute.
\begin{center}
\begin{tikzcd}[column sep=4em, row sep=3em]
\Gamma^n \Gamma^n \Gamma^n \arrow[r, "\Gamma^n (\bar \eta_{\Gamma^n})^{-1}"] \arrow[d, "((\bar \eta_{\Gamma^n})^{-1})_{\Gamma^n}"'] & \Gamma^n \Gamma^n \arrow[d, "(\bar \eta_{\Gamma^n})^{-1}"] \\
\Gamma^n \Gamma^n \arrow[r, "(\bar \eta_{\Gamma^n})^{-1}"]                                           & \Gamma^n                                   
\end{tikzcd}
\begin{tikzcd}[column sep=4em, row sep=3em]
\Gamma^n \arrow[r, "\bar \eta_{\Gamma^n}"] \arrow[rd, "="'] & \Gamma^n \Gamma^n \arrow[d, "(\bar \eta_{\Gamma^n})^{-1}"] & \Gamma^n \arrow[l, "\Gamma^n \bar \eta"'] \arrow[ld, "="] \\
                                                  & \Gamma^n                                         &                                                
\end{tikzcd}
\end{center}
The first diagram commutes because
\begin{align*}
    &(\bar \eta_{\Gamma^n})^{-1} \circ ((\bar \eta_{\Gamma^n})^{-1})_{\Gamma^n}
    = (\bar \eta_{\Gamma^n})^{-1} \circ (\bar \eta_{\Gamma^n \Gamma^n})^{-1}
    = \left( \bar \eta_{\Gamma^n\Gamma^n} \circ \bar \eta_{\Gamma^n}\right)^{-1} \\
    = &\left( \Gamma^n \bar \eta_{\Gamma^n} \circ \bar \eta_{\Gamma^n}\right)^{-1}
    = ( \bar \eta_{\Gamma^n})^{-1} \circ (\Gamma^n \bar \eta_{\Gamma^n})^{-1}
    = ( \bar \eta_{\Gamma^n})^{-1} \circ \Gamma^n (\bar \eta_{\Gamma^n})^{-1}.
\end{align*}
The second diagram commutes because
\begin{equation*}
    (\bar \eta_{\Gamma^n})^{-1} \circ \Gamma^n \bar \eta
    = (\bar \eta_{\Gamma^n})^{-1} \circ \bar \eta_{\Gamma^n}
    = \id.
\end{equation*}
\end{proof}

\begin{theorem}[Theorem A]\label{thm:theoremA}
    Let $P$ be a distributive lattice and $\A$ an abelian category. If $F \colon P \to \A$ is a functor, then
    \begin{itemize}
        \item $\Gamma_n F$ is cross-codegree $n$, and
        \item $\Gamma^n F$ is cross-degree $n$.
    \end{itemize}
\end{theorem}

\begin{proof}
    We prove the first part. The second is dual.

    By \autoref{lemma:cross_effect_A}, it suffices to show that $\Cr_n (\Gamma_n F) \cong 0$.
    By definition, $\Gamma_n F = \Ima(T_n F \to F)$. Recall that there is a canonical natural epimorphism $T_n F \to \Gamma_n F$. By naturality of $T_n \to \id$, we get a commutative diagram,
    \begin{center}
    \begin{tikzcd}
    T_n T_n F \arrow[d, "\cong"] \arrow[r] & T_n \Gamma_n F. \arrow[d] \\
    T_n F \arrow[r, two heads]             & \Gamma_n F,
    \end{tikzcd}
    \end{center}
    where the left hand vertical map is an isomorphism.
    It follows that the map $T_n \Gamma_n F \to \Gamma_n F$ is an epimorphism. Hence, $\Cr_n (\Gamma_n F) \cong \coker(T_n \Gamma_n F \to \Gamma_n F) \cong 0$.
\end{proof}

\begin{theorem}[Theorem B]\label{thm:theoremB}
    Let $P$ be a join-factorization lattice, $\A$ an abelian category and $F \colon P \to \A$ a functor.
    If $F$ is cross-codegree $n$, then $\Gamma_n F \cong F$.

    Dually,
    let $P$ be a meet-factorization lattice, $\A$ an abelian category and $F \colon P \to \A$ a functor.
    If $F$ is cross-degree $n$, then $\Gamma^n F \cong F$.
\end{theorem}

\begin{proof}
    We prove the first part. The second is dual.

    Suppose $F$ is cross-codegree $n$. Then, by \autoref{lemma:cross_effect_B}, $\coker(T_n F \to F)\cong  \Cr_n F \cong 0$. Hence, the map $T_n F \to F$ is an epimorphism. Thus, $\Gamma_n F = \Ima(T_n F \to F) \cong F$, as desired.
\end{proof}

\begin{example}
    We illustrate \autoref{thm:theoremA} and \autoref{thm:theoremB} in the case $n = 0$.
    Let $P$ be a join-factorization lattice, and let $\widehat 0$ denote its least element. 
    Let further $\A$ be an abelian category and $F \colon P \to \A$ a functor. Then, as $\widehat 0$ is the sole element in $P$ of join-dimension 0,
    \begin{equation*}
        \Gamma_0 F(x) \cong \Ima (T_0 F(x) \to F(x))
        \cong \Ima(F(\widehat 0) \to F(x)) \cong \Ima F(\widehat 0 \le x).
    \end{equation*}
    In other words, $\Gamma_0 F$ is the functor $\Ima F(\widehat 0 \le -)$.
    Note further that $F$ is cross-codegree 0 if and only if $\coker F(x \le y) \cong 0$ for all $x \le y$, or, in other words, if $F$ sends every morphism in $P$ to an epimorphism.
    \autoref{thm:theoremA} and \autoref{thm:theoremB} thus tells us that $F$ sends every morphism to an epimorphism if and only if 
    \begin{equation*}
        F \cong \Ima F(\widehat 0 \le -).
    \end{equation*}

    The approximation $\Gamma^0 F$ captures the dual concept: if $\widehat 1$ is a greatest element in $P$, then $\Gamma^0 F \cong \Ima F(- \le \widehat 1)$, and a cross-degree 0 functor is one that sends every morphism to a monomorphism. 
\end{example}

We introduce the following notation. Given a lattice $P$ and an abelian category $\A$,
let $\Fun_{\Cr_n = 0}(P, \A)$ denote the full subcategory of cross-codegree $n$ functors in $\Fun(P, \A)$, and let $i_n \colon \Fun_{\Cr_n = 0}(P, \A) \hookrightarrow \Fun(P, \A)$ denote the corresponding inclusion functor.
Likewise, let $\Fun_{\Cr^n = 0}(P, \A)$ denote the full subcategory of cross-degree $n$ functors, and let $i^n \colon \Fun_{\Cr^n = 0}(P, \A) \hookrightarrow \Fun(P, \A)$ denote the corresponding inclusion functor.

\begin{proposition}\label{prop:adjoints}
    For every join-factorization lattice $P$ and abelian category $\A$,
   $\Gamma_n$ is the right adjoint of the inclusion $i_n \colon \Fun_{\Cr_n=0}(P, \A) \to \Fun(P, \A)$
    
    Dually, for every meet-factorization lattice $P$ and abelian category $\A$,
    $\Gamma^n$ is the left adjoint of the inclusion $i^n \colon \Fun_{\Cr^n=0}(P, \A) \to \Fun(P, \A)$
\end{proposition}

\begin{proof}
    We prove the second part. The first is dual.
    By \cite[4.2.3]{handbook_vol2}, an idempotent monad is a reflector onto its essential image (i.e., its essential image is a full reflective subcategory, and the monad is its left adjoint). By \autoref{prop:idempotent_monad}, $\Gamma^n$ admits the structure of an idempotent monad. It follows from \autoref{thm:theoremA} and \autoref{thm:theoremB} that the essential image of $\Gamma^n$ is $\Fun_{\Cr^n = 0}(P, \A)$.
\end{proof}

\begin{proposition}\label{prop:universal_prop_cross}
    Let $P$ be a join-factorization lattice and $\A$ an abelian category. For any functor $F \colon P \to \A$, every natural transformation $\alpha \colon G \to F$ from a cross-codegree $n$ functor $G$ factors uniquely through $\Gamma_nF \to F$.
    
    Dually, let $P$ be a meet-factorization lattice and $\A$ an abelian category.  For any functor $F \colon P \to \A$, every natural transformation $\alpha \colon F \to G$ to a cross-degree $n$ functor $G$ factors uniquely through $F \to \Gamma^n F$.
\end{proposition}

\begin{proof}
    This follows directly from \autoref{prop:adjoints} and the universal morphism property of adjoint functors.
\end{proof}

\begin{proposition}
    Let $P$ be a join-factorization lattice and $\A$ an abelian category.
     For any functor $F \colon P \to \A$, the cotower of cross-codegree $n$ approximations of $F$ converges to $F$.

    Dually,
    let $P$ be a meet-factorization lattice and $\A$ an abelian category.
     For any functor $F \colon P \to \A$, the tower of cross-degree $n$ approximations of $F$ converges to $F$.
\end{proposition}

\begin{proof}
    We prove the first part. The second is dual.
    
    By \cite[Proposition 4.28]{hem2025posetfunctorcocalculusapplications}, the morphism
    \begin{equation*}
        \colim_n T_n F \to F
    \end{equation*}
    is an isomorphism. This morphism factors as
    \begin{equation}\label{eq:colimn_Tn_Pn}
        \colim_n T_n F \to \colim_n \Gamma_n F\to F.
    \end{equation}
    In particular,
    \begin{equation*}
        \colim_n T_n F \to \colim_n \Gamma_n F
    \end{equation*}
    is a split monomorphism.
    Furthermore,
    \begin{align*}
        &\coker \left( \colim_n T_n F \to \colim_n \Gamma_n F \right) \\
        \cong 
        &\colim_n \coker (T_n F \to \Gamma_n F ) \cong \colim_n 0 \cong 0.
    \end{align*}
    Thus, both morphisms in \eqref{eq:colimn_Tn_Pn} are isomorphisms.
\end{proof}

\subsection{Properties of cross-(co)degree $n$ approximations}

While $T_n$ commutes with colimits in $\Fun(P, \A)$, this is \emph{not} the case for $\Gamma_n$. Likewise, $\Gamma^n$ does not, in general, commute with limits.
However, we show here that $\Gamma_n$ enjoys many interesting properties that $T_n$ does not.

\begin{example}[$\Gamma_n$ does not commute with colimits.]
Let $\Fb$ be a field.
Let $F \colon \{0,1\}^2 \to \Vecs_{\Fb}$ be defined as
\begin{center}
    \begin{tikzcd}
    0 \arrow[r]           & \Fb           \\
    0 \arrow[u] \arrow[r] & 0 \arrow[u].
    \end{tikzcd}
\end{center}
Furthermore,
let $G \colon \{0,1\}^2 \to \Vecs_{\Fb}$ be defined as
\begin{center}
    \begin{tikzcd}
    \Fb \arrow[r, "1"]           & \Fb           \\
    \Fb \arrow[u, "1"] \arrow[r, "1"] & \Fb \arrow[u, "1"].
    \end{tikzcd}
\end{center}
Let further $\alpha \colon F \to G$ be the obvious inclusion map.

Then, $\Gamma_1 F \cong 0$ and $\Gamma_1 G \cong G$, so $\coker(\Gamma_1
\alpha) \cong G$.
Moreover,
$\coker(\alpha)$ is
\begin{center}
    \begin{tikzcd}
    \Fb \arrow[r]           & 0           \\
    \Fb \arrow[u, "1"] \arrow[r, "1"] & \Fb \arrow[u],
    \end{tikzcd}
\end{center}
which is cross-codegree $1$ and different from $G$. Thus, $\coker(\Gamma_1 \alpha) \ncong \Gamma_1 \coker(\alpha)$.
\end{example}

\begin{proposition}\label{prop:preserve_mono_epi}
    Let $P$ be a distributive lattice and $\A$ an abelian category. Let $F, G \colon P \to \A$ be two functors, and let $\alpha \colon F \to G$ be a natural transformation.
    \begin{enumerate}
        \item[(i)] If $\alpha \colon F \to G$ is a monomorphism, then both $\Gamma_n \alpha$ and $\Gamma^n \alpha$ are monomorphisms.
        \item[(ii)] If $\alpha \colon F \to G$ is an epimorphism, then both $\Gamma_n \alpha$ and $\Gamma^n \alpha$ are epimorphisms.
    \end{enumerate}
\end{proposition}

\begin{proof}
    We prove (i). The proof of (ii) is similar.

    Let $\alpha \colon F \to G$ be a monomorphism. We first prove that $\Gamma_n \alpha$ is a monomorphism.
    Consider the commutative diagram, where the vertical maps are the canonical monomorphisms.
    \begin{center}
    \begin{tikzcd}
        \Gamma_n F \arrow[r, "\Gamma_n \alpha"] \arrow[d, tail] & \Gamma_n G \arrow[d, tail] \\
        F \arrow[r, "\alpha", tail]                   & G                    .
    \end{tikzcd}
    \end{center}
    As the lower-left composition is a monomorphism, the top-right composition is a monomorphism. Thus, $\Gamma_n \alpha$ is a monomorphism.

    We now show that $\Gamma^n \alpha$ is a monomorphism. As $T^n$ preserves limits, and thus kernels, $\ker (T^n \alpha) = T^n \ker (\alpha) = T^n 0 = 0$. Thus, $T^n \alpha$ is a monomorphism. Now, consider the commutative diagram, where the vertical maps are the canonical monomorphisms.
    \begin{center}
    \begin{tikzcd}
        \Gamma^n F \arrow[r, "\Gamma^n \alpha"] \arrow[d, tail] & \Gamma^n G \arrow[d, tail] \\
        T^nF \arrow[r, "T^n\alpha", tail]                   & T^n G                    .
    \end{tikzcd}
    \end{center}
    Again, as the lower-left composition is a monomorphism, the top-right composition is a monomorphism. Thus, $\Gamma^n \alpha$ is a monomorphism.
\end{proof}

\begin{proposition}
    Let $P$ be a distributive lattice and $\A$ an abelian category.
    Let $F, G \colon P \to \A$. Then,
    \begin{itemize}
        \item $\Gamma_n(F \oplus G) \cong \Gamma_nF \oplus \Gamma_n G$, and
        \item $\Gamma^n(F \oplus G) \cong \Gamma^n F \oplus \Gamma^n G$.
    \end{itemize}
\end{proposition}

\begin{proof}
    Firstly, as $T_n$ preserves coproducts and $T^n$ preserves products, both $T_n$ and $T^n$ preserves finite direct sums in an abelian category. Furthermore, the image of a direct sum of a map is the direct sum of the images of the maps, which concludes the proof.
\end{proof}

\begin{theorem}[Distributive law]\label{thm:dist_law}
    Let $P$ be a distributive lattice, $\A$ an abelian category and $F \colon P \to \A$ a functor. For any $m,n \ge 0$,
    \begin{equation*}
        \Gamma_m \Gamma^n F \cong \Gamma^n \Gamma_m F.
    \end{equation*}
    Furthermore, this isomorphism is natural in $F$.
\end{theorem}

\begin{proof}
    Let $\bar \varepsilon$ denote the canonical natural transformation $\Gamma_n \to \id_{\Fun(P,\A)}$, and $\bar \eta$ the canonical natural transformation $\id_{\Fun(P, \A)} \to \Gamma^m$. Recall that $\bar \varepsilon$ is an objectwise monomorphism, and $\bar \eta$ is an objectwise epimorphism.

    Considering the natural transformation $\Gamma_n \bar \eta$ together with the morphism $\bar \varepsilon_F$ gives the following commutative diagram.
    \begin{center}
        \begin{tikzcd}[row sep=4em, column sep=4em]
        \Gamma_n \Gamma_n F \arrow[r, "\Gamma_n \bar \varepsilon_F", "\cong"'] \arrow[d, "\Gamma_n \bar \eta_{\Gamma_n F}"', two heads] & \Gamma_n F \arrow[d, "\Gamma_n \bar\eta_F", two heads] \\
        \Gamma_n \Gamma^m \Gamma_n F \arrow[r, "\Gamma_n \Gamma^m \bar \varepsilon_F", tail]                              & \Gamma_n \Gamma^m F                                
        \end{tikzcd}
    \end{center}
    By \autoref{prop:idempotent_monad}, $\Gamma_n \bar \varepsilon_F$ is an isomorphism.
    Furthermore, by \autoref{prop:preserve_mono_epi}, $\Gamma_n \bar \eta_F$ is an epimorphism.
    Hence, $\Gamma_n \Gamma^m \bar \varepsilon_F$ is an epimorphism. Thus, as $\Gamma_n \Gamma^m \bar \varepsilon_F$ is also a monomorphism (again, by \autoref{prop:preserve_mono_epi}), $\Gamma_n \Gamma^m \bar \varepsilon_F$ is an isomorphism.

    Now, considering the natural transformation $\bar \varepsilon$ together with the morphism $\bar \eta_{\Gamma_n F}$ gives the following commutative diagram.
    \begin{center}
        \begin{tikzcd}[row sep=4em, column sep=4em]
        \Gamma_n \Gamma_n F \arrow[r, "\Gamma_n \bar \eta_{\Gamma_n F}", two heads] \arrow[d, "\bar \varepsilon_{\Gamma_n F}"', "\cong"] & \Gamma_n \Gamma^m \Gamma_n F \arrow[d, "\bar \varepsilon_{\Gamma^m \Gamma_n F}", tail] \\
        \Gamma_n F \arrow[r, "\bar \eta_{\Gamma_n F}", two heads]                              & \Gamma^m \Gamma_n F                                
        \end{tikzcd}
    \end{center}
    By, \autoref{prop:idempotent_monad}, $\bar \varepsilon_{\Gamma_n F}$ is an isomorphism. Thus, as $\bar \eta_{\Gamma_n F}$ is an epimorphism, $\bar \varepsilon_{\Gamma^m \Gamma_n F}$ is an epimorphism. Thus, as $\bar \varepsilon_{\Gamma^m \Gamma_n F}$ is also a monomorphism, and $\A$ is an abelian category, $\bar \varepsilon_{\Gamma^m \Gamma_n F}$ is an isomorphism.

    In conclusion,
    \begin{equation*}
        \Gamma_n \Gamma^m F \cong \Gamma_n \Gamma^m \Gamma_n F \cong \Gamma^m \Gamma_n F.
    \end{equation*}
    Naturality follows directly from the fact that each of these two isomorphisms is natural.
\end{proof}

\section{Application to multipersistence}

In this section, we will prove two theorems that reveal a connection between cross-degree and projective dimension of persistence modules.
We then give examples where we compute the cross-(co)degree of persistence modules, and in particular showcase situations where multifiltrations carry extra structure that imply sharper bounds for the cross-(co)degree of their homology.

First, we review some preliminaries on finite distributive lattices and on minimal projective resolutions of persistence modules.

\subsection{Combinatorics of finite distributive lattices}

For two elements $x,y$ in a poset $P$, we say that $y$ \emph{covers} $x$ if $y > x$ and that there exists no $z$ with $y > z > x$. We denote this by $y \succ x$. If $y$ covers $x$, we say that $x$ is a \emph{parent} of $y$ and that $y$ is a \emph{child} of $x$.

\begin{lemma}\label{lemma:join_of_parents}
    Let $P$ be a lattice, and let $u \in P$. If $a,b \prec u$ with $a \neq b$, then $a \vee b = u$.

    Dually, if $a,b \succ u$ with $a \neq b$, then $a \wedge b = u$.
\end{lemma}
\begin{proof}
    The first part is proved in \cite[Lemma 7.1]{hem2025decomposingmultipersistencemodulesusing}. The second is dual.
\end{proof}

\begin{lemma}\label{lemma:meet_of_cover}
    Let $P$ be a distributive lattice, and $a,x,y \in P$.
    
    If $y \succ a$ and $a < x \ngeq y$, then $x \vee y \succ x$.
    
    Dually, if $a \succ y$ and $a > x \nleq y$, then $x \succ x \wedge y$.
\end{lemma}
\begin{proof}
    This is stated for modular lattices in \cite[Corollary V.2]{Birkhoff}, and all distributive lattices are modular.
\end{proof}

\begin{lemma}\label{lemma:num_of_parents}
    Let $P$ be a finite distributive lattice. If $x \in P$, then $x$ has exactly $\jdim(x)$ parents and exactly $\mdim(x)$ children
\end{lemma}
\begin{proof}
    The first part follows from \cite[Lemma 4.20]{hem2025posetfunctorcocalculusapplications} and \cite[Proposition 5.8]{realisationsposetstameness}.
    The proof of the second part is dual.
\end{proof}

In the following definition, we adopt the convention that $[-1] = \emptyset$, so that a 0-cube in a poset $P$ is the same as an element in $P$.
\begin{definition}
    Let $P$ be a finite lattice, and $a \in P$. Let $n = \jdim(a)$, and let $x_0, \dots, x_{n-1}$ be the parents of $a$. We define the \emph{parent-cube} of $a$ as the $n$-cube $\PC_a \colon \Po([n-1]) \to \Vecs_{\Fb}$ defined by
    \begin{equation*}
        \PC_a(S) = \begin{cases}
            a, & S = [n-1], \\
            \bigwedge_{i \notin S} x_i, & \textrm{otherwise.}
        \end{cases}
    \end{equation*}
    Dually, let $n = \mdim(a)$, and let $y_0, \dots, y_{n-1}$ be the children of $a$.
    We define the \emph{child-cube} of $a$ as the $n$-cube $\CC_a \colon \Po([n-1]) \to \Vecs_{\Fb}$ defined by
    \begin{equation*}
        \CC_a(S) = \begin{cases}
            a, & S = \emptyset, \\
            \bigvee_{i \in S} y_i, & \textrm{otherwise.}
        \end{cases}
    \end{equation*}
\end{definition}

Note that in the case where $a \in P$ has 0 parents, the parent-cube of $a$ is the 0-cube in $a$.

\begin{lemma}\label{lemma:pc_bicart}
    In a distributive lattice, all parent-cubes and child-cubes are strongly bicartesian.
\end{lemma}
\begin{proof}
    We prove the statement for parent-cubes. The statement for child-cubes is dual.

    Let $P$ be a finite distributive lattice, and $a \in P$. Let $n = \jdim(a)$ (which equals the number of parents of $a$, by \autoref{lemma:num_of_parents}). If $n=0$ or $n=1$, the statement follows trivially. 
    
    Suppose that $n \ge 2$. Let $x_0, \dots, x_{n-1}$ be the parents of $a$. By \autoref{lemma:join_of_parents}, $x_i \vee x_j = a$ for all $i \neq j$.
    It now follows from \cite[Lemma 3.6]{hem2025posetfunctorcocalculusapplications} that $\PC_a$ is strongly bicartesian.
\end{proof}

%\subsection{The dimension of a lattice}

Given posets $P$ and $Q$, we say that $P$ \emph{embeds} into $Q$ if there exists an order-preserving, injective function $i \colon P \hookrightarrow Q$.

\begin{definition}
    Let $P$ be a poset. The \emph{dimension}\footnote{also called \emph{order-dimension} or \emph{Dushnik–Miller dimension}} of $P$ is defined as the minimal number of total orders such that P admits an embedding into their product.
\end{definition}

\begin{proposition}\label{prop:order_dim_jdim}
    If $P$ is a finite distributive lattice, then the dimension of $P$ equals
    \begin{equation*}
        \max\{\jdim(x) : x \in P\}.
    \end{equation*}
\end{proposition}
\begin{proof}
    It is stated in \cite[Theorem 1.2]{dilworth_theorem} that the dimension of a finite distributive lattice $P$ equals the maximum $k$ such that there exists an element $x \in P$ with $k$ parents.
    The statement now follows from \autoref{lemma:num_of_parents}.
\end{proof}

\begin{corollary}
    Let $P$ be a finite distributive lattice. Then,
    \begin{equation*}
        \max\{\jdim(x) : x \in P\} = \max\{\mdim(x) : x \in P\}.
    \end{equation*}
\end{corollary}
\begin{proof}
    The equation can be rewritten as
    \begin{equation*}
        \max\{\jdim(x) : x \in P\} = \max\{\jdim(x) : x \in P^{\op}\}.
    \end{equation*}
    By \autoref{prop:order_dim_jdim}, the left-hand side is the dimension of $P$ and the right-hand size is the dimension of $P^{\op}$. These numbers are the same, as if $P$ embeds into $T_1 \times \dots \times T_n$, then $P^{\op}$ embeds into $T_1^{\op} \times \dots \times T_n^{\op}$.
\end{proof}

\begin{example}
We give the dimension of some finite distributive lattices.
\begin{itemize}
    \item The poset $[m_1] \times \dots \times [m_n]$, i.e., a product of $n$ finite total orders, has dimension $n$ (assuming $m_i > 0$ for all $i$). 
    \item For a finite set $V$, the poset $\Po(V)$ has dimension $|V|$ (as it's isomorphic to $\{0,1\}^{|V|}$).
\end{itemize}
\end{example}

\subsection{Persistence modules and minimal resolutions}

A \emph{persistence module} is a functor of the form $F \colon P \to \Vecs_{\Fb}$, where $P$ is a poset and $\Vecs_{\Fb}$ is the category of vector spaces over some field $\Fb$.
We will use the term \emph{$n$-parameter multipersistence module} when referring to a persistence module whose source poset is a product of $n$ total orders. It is common to refer to 2-parameter multipersistence modules as \emph{bipersistence modules}.

For any poset $P$, the category $\Fun(P, \Vecs_{\Fb})$ is an abelian category with enough projectives, and every projective in this category is free \cite[Proposition 5]{projectiveDiagsFree}, i.e., of the form
\begin{equation*}
    \bigoplus_{x \in P} V_x [x, -),
\end{equation*}
where each $V_x$ is a vector space, and $V_x [x, -)$ denotes the functor $\Lan_{\{x\} \hookrightarrow P} V_x$.

A \emph{minimal resolution} of a functor $F \colon P \to \Vecs_{\Fb}$ is a free resolution
\begin{equation*}
\dots \to Q_1 \to Q_0 \xrightarrow{p} F,
\end{equation*}
such that any chain map $\phi \colon Q_{\bullet} \to Q_{\bullet}$ satisfying $p \circ \phi = p$ is an isomorphism. In \cite[Section 10]{realisationsposetstameness}, it is proved that minimal resolutions exist in the setting where the poset $P$ is finite, and a formula for computing them using Koszul complexes is given. We present this formula here, in the special case where the source poset is a finite distributive lattice.

Given a persistence module $F \colon P \to \Vecs_{\Fb}$, we let $\{(\beta^i F)_x : i \in \Nn_{\ge 0}, x \in P\}$ denote the vector spaces so that the $i$th term in the minimal resolution of $F$ is
\begin{equation*}
    Q_i = \bigoplus_{a \in P} (\beta^iF)_a [a, -).
\end{equation*}
We call $\{(\beta^iF)_a : a \in P\}$ the \emph{$i$th Betti diagram} of $F$.

We recall here the usual definition of the Koszul complex, as formulated in \cite{realisationsposetstameness}.
\begin{definition}
    Given a $k$-cube $\X \colon \Po([k-1]) \to \Vecs_{\Fb}$, the \emph{Koszul complex} of $\X$, denoted $K_{\X}$, is the chain complex in $\Vecs_{\Fb}$ given by
    \begin{equation*}
        (K_{\X})_i = \bigoplus_{S \subseteq [k-1], |S| = k-i} \X (S),
    \end{equation*}
    with differential $\partial_{i+1} \colon (K_{\X})_{i+1} \to (K_{\X})_{i}$ defined componentwise by
    \begin{equation*}
        \partial_{i+1}|_{\X(S)}
        = \sum_{j=0}^i (-1)^j \X (S \subseteq S \cup \{t_j\}),
    \end{equation*}
    where $t_0 < \dots < t_i$ are the elements in $[k-1] \setminus S$.
\end{definition}

\begin{theorem}\label{thm:betti_diag}
    Let $P$ be a finite distributive lattice, and $F \colon P \to \Vecs_{\Fb}$ a finitely generated persistence module.
    For all $a \in P$,
    \begin{equation*}
        (\beta^i F)_a \cong H_i (K_{F \circ \PC_a}).
    \end{equation*}
\end{theorem}
\begin{proof}
    This is \cite[Theorem 10.17]{realisationsposetstameness}, in the special case where the source poset is a distributive lattice.
\end{proof}

The following is Lemma 8.15 in \cite{hem2025decomposingmultipersistencemodulesusing}.
\begin{lemma}\label{lemma:cube_colimits_middle_exact}
    Let $\X \colon \Po([k-1]) \to \Vecs_{\Fb}$ be a $k$-cube. Then
    \begin{equation*}
        \underset{{S \subsetneq [k-1]}}{\colim} \X(S) \cong \coker\left((K_\X)_2 \xrightarrow{\partial_2} (K_\X)_1\right),
    \end{equation*}
    and
    \begin{equation*}
        \lim_{\emptyset \subsetneq S \subseteq [k-1]} \X(S) \cong \ker\left((K_\X)_{k-1} \xrightarrow{\partial_{k-1}} (K_\X)_{k-2}\right).
    \end{equation*}
\end{lemma}

We say that an object $A$ in an abelian category $\A$ has projective dimension at most $n$ if it admits a projective resolution of the following form:
\begin{equation*}
    0 \to Q_n \to \dots \to Q_0 \to A.
\end{equation*}
\begin{lemma}\label{lemma:meet_preserving_projectives}
    Let $P$ be a lattice and $P'$ a finite lattice. If $f \colon P' \to P$ is a functor that preserves meet, then the functor
    \begin{equation*}
        f^* \colon \Fun(P, \Vecs_{\Fb}) \to \Fun(P', \Vecs_{\Fb}): F \mapsto F\circ f
    \end{equation*}
    preserves projective objects.

    In particular, if the projective dimension of $F \colon P \to \Vecs$ is $n$, then the projective dimension of $F \circ f \colon P' \to \Vecs_{\Fb}$ is at most $n$.
\end{lemma}
\begin{proof}
    We use the standard fact that a functor preserves projective objects if it has a right adjoint that preserves epimorphisms (\cite[Proposition 2.3.10]{weibel1994introduction}, \cite[Proposition 7.2.17]{hirschhorn}).

    The right adjoint of $f^*$ is $\Ran_f$. Given a functor $F \in \Fun(P', \Vecs_{\Fb})$, $\Ran_f F$ is given by (\cite[6.2]{riehl_category_theory_in_context})
    \begin{align*}
        (\Ran_f F)(x) \cong \lim_{y \in P', f(y) \ge x} F(y).
    \end{align*}
    Now, as the set $\{y \in P':f(y) \ge x\}$ is either empty or has a least element,
    \begin{align*}
        (\Ran_f F)(x) \cong \begin{cases}
            F(\bigwedge_{y \in P', f(y) \ge x} y), \quad & \{y \in P':f(y) \ge x\} \neq \emptyset, \\
            0, \quad &\textrm{otherwise.}
        \end{cases}
    \end{align*}
    Furthermore, $\Ran_f$ sends natural transformations to the obvious induced map.
    Because epimorphisms in $\Fun(P', \Vecs_{\Fb})$ are those natural transformations that are objectwise epimorphisms, it follows that if $\eta \colon F \to G$ is an epimorphism in $\Fun(P', \Vecs_{\Fb})$, then $\Ran_f(\eta)$ is also an epimorphism.

    For the final part, it suffices to note that as $f^*$ is exact and preserves projectives, it sends a projective resolution of $F$ of length $n$ to a projective resolution of $f^*(F)$ of length $\le n$.
\end{proof}

\subsection{Connections to projective dimension}

We show how cross-degree is related to projective dimension.

We first prove a lemma that is useful for computations of total fibers and cofibers.

\begin{lemma}\label{lemma:koszul_tfib}
    Let $\X \colon \Po([k-1]) \to \Vecs_{\Fb}$ be a $k$-cube with associated Koszul complex $K_{\X}$. Then
    \begin{equation*}
        \tcofib \X \cong H_0(K_{\X})
    \end{equation*}
    and
    \begin{equation*}
        \tfib \X \cong H_k(K_{\X}).
    \end{equation*}
\end{lemma}
\begin{proof}
    We prove the first part. The second is dual.
    By definition,
    \begin{equation*}
        H_0(K_{\X}) = \coker \left( \bigoplus_{S \subseteq [k-1], |S| = k-1} \X(S) \to \X([k-1]) \right).
    \end{equation*}
    By \autoref{lemma:tocofib}, this is isomorphic to $\tcofib(\X)$.
\end{proof}

\begin{theorem}\label{thm:pdim_thm1}
    Let $P$ be a finite distributive lattice of dimension $n$, with $n \ge 1$, and $F \colon P \to \Vecs_{\Fb}$ a finitely generated persistence module. Then the following are equivalent.
    \begin{enumerate}
        \item $F$ has projective dimension at most $n-1$.
        \item $F$ is cross-degree $n-1$.
        \item $F \cong \Gamma^{n-1} F$.
    \end{enumerate}
\end{theorem}

\begin{proof}
    By \autoref{thm:theoremB}, $(2) \implies (3)$, and by \autoref{thm:theoremA}, $(3) \implies (2)$.

    We show $(2) \implies (1)$. Suppose that $F$ is cross-degree $n-1$. We need to show that $(\beta^n F)_a \cong 0$ for every $a \in P$. By \autoref{thm:betti_diag}, it suffices to show that $H_n (K_{F \circ \PC_a}) \cong 0$ for every $a \in P$. Let $a \in P$.
    If $\jdim(a) < n$, then $(K_{F \circ \PC_a})_n = 0$, so we are done.
    If $\jdim(a) = n$, then
    \begin{equation*}
         H_n(K_{F \circ \PC_a}) \cong \tfib (F \circ \PC_a),
    \end{equation*}
    by \autoref{lemma:koszul_tfib}. Now, as $\PC_a$ is strongly bicartesian, and $F$ is cross-degree $n-1$ by assumption, $\tfib(F \circ \PC_a) \cong 0$, as desired.

    Finally, we show that $(1) \implies (2)$. Suppose that $F$ has projective dimension at most $n-1$, and let $\X \colon \Po([n-1]) \to P$ be a strongly bicartesian $n$-cube.
    Then $F \circ \X$ also has projective dimension $\le n-1$, by \autoref{lemma:meet_preserving_projectives} (recall from \autoref{rmk:cubes_bicart_meet} that $\X$ preserves meets). 
    Then, again applying \autoref{lemma:koszul_tfib} and \autoref{thm:betti_diag},
    \begin{equation*}
        \tfib(F \circ \X) \cong H_n (K_{F \circ \X}) \cong 0.
    \end{equation*}
\end{proof}

\begin{theorem}\label{thm:pdim_thm2}
    Let $P$ be a finite distributive lattice of dimension $n$, with $n \ge 2$, and $F \colon P \to \Vecs_{\Fb}$ a finitely generated persistence module. Then the following are equivalent.
    \begin{enumerate}
        \item $F$ has projective dimension at most $n-2$.
        \item $F$ is degree $n-1$ and cross-degree $n-2$.
        \item $F \cong T^{n-1} \Gamma^{n-2} F$.
    \end{enumerate}
\end{theorem}

\begin{proof}
\begin{description}
    \item[$(2) \implies (3)$] As $F$ is cross-degree $n-2$, $\Gamma^{n-2} F \cong F$, by \autoref{thm:theoremB} (as any finite distributive lattice is a meet-factorization lattice). Similarly, as $F$ is degree $n-1$, $T^{n-1} F \cong F$, by (the dual of) \cite[Theorem B]{hem2025decomposingmultipersistencemodulesusing}. Thus, $T^{n-1} \Gamma^{n-2} F \cong T^{n-1} F \cong F$.
    \item[$(3) \implies (2)$] By \autoref{thm:theoremA}, $\Gamma^{n-2} F$ is cross-degree $n-2$. Hence, $T^{n-1} \Gamma^{n-2} F$ is cross-degree $n-2$, by \autoref{lemma:Tn_preserve_cross_degree}. Furthermore, $T^{n-1} \Gamma^{n-2} F$ is degree $n-1$ by (the dual of) \cite[Theorem A]{hem2025decomposingmultipersistencemodulesusing}.
    \item[$(2) \implies (1)$] By \autoref{thm:betti_diag}, it suffices to show that for each $a \in P$, the Koszul complex $K_{F \circ \PC_a}$ has 0 homology in degree $n-1$. There are three cases to consider:
    \begin{itemize}
        \item $\jdim(a) < n-1$: In this case, then $(K_{F \circ \PC_a})_i = 0$ for $i \ge n-1$, so the homology in these degrees are trivially 0.
        \item $\jdim(a) = n-1$: In this case, $\PC_a$ is a strongly bicartesian $(n-1)$-cube, and $H_{n-1} (K_{F \circ \PC_a}) \cong \tfib (F \circ \PC_a)$, by \autoref{lemma:koszul_tfib}.
        As $F$ is cross-degree $n-2$, $\tfib (F \circ \PC_a) \cong 0$.
        \item $\jdim(a) = n$: In this case, $\PC_a$ is a strongly bicartesian $n$-cube. By \autoref{lemma:cube_colimits_middle_exact}, the map
        \begin{equation*}
            (K_{F \circ \PC_a})_n \to \ker\left((K_{F \circ \PC_a})_{n-1} \to (K_{F \circ \PC_a})_{n-2} \right)
        \end{equation*}
        is naturally isomorphic to
        \begin{equation*}
            (F \circ \PC_a)(\emptyset) \to \lim_{\emptyset \subsetneq S \subseteq [n-1]} F\circ \PC_a,
        \end{equation*}
        which is an isomorphism as $F$ is degree $n-1$.
        Thus, $H_{n-1} (K_{F \circ \PC_a}) \cong H_{n} (K_{F \circ \PC_a}) \cong 0$.
    \end{itemize}
    \item[$(1) \implies (2)$] 
        Suppose that $F$ has projective dimension $n-2$. We first show that $F$ is cross-degree $n-2$.
        Let $\X \colon \Po([n-2]) \to P$ be a strongly bicartesian $(n-1)$-cube.
        Then $F \circ \X$ also has projective dimension $n-2$, by \autoref{lemma:meet_preserving_projectives}. 
        Hence,
        \begin{equation*}
            \tfib(F \circ \X) \cong H_{n-1} (K_{F \circ \X}) \cong 0.
        \end{equation*}

        We now show that $F$ is degree $n-1$.
        Let $\X \colon \Po([n-1]) \to P$ be a strongly bicartesian $n$-cube.
        Then $F \circ \X$ also has projective dimension $n-2$, again by \autoref{lemma:meet_preserving_projectives}.
        Hence, $H_n (K_{F \circ \X}) \cong H_{n-1} (K_{F \circ \X}) \cong 0$, so the map
        \begin{equation*}
            (K_{F \circ \X})_n \to \ker \big( (K_{F \circ \X})_{n-1} \to (K_{F \circ \X})_{n-2} \big)
        \end{equation*}
        is an isomorphism. Hence, by \autoref{lemma:cube_colimits_middle_exact}, $F\circ \X$ is cartesian.
\end{description}
\end{proof}

\begin{remark}
    It is not known to the author whether the order of $T^{n-1}$ and $\Gamma^{n-2}$ can be swapped in the preceding theorem, i.e., whether $\Gamma^{n-2} T^{n-1} F$ is always isomorphic to $T^{n-1} \Gamma^{n-2} F$.
\end{remark}

\begin{example}[Non-example]
    We show that the conditions on the poset dimension in \autoref{thm:pdim_thm1} and \autoref{thm:pdim_thm2} are necessary.
    This example is inspired by \cite[Example 15]{lebovici2024localcharacterizationblockdecomposabilitymultiparameter}.
    Consider the following multipersistence module over $\{0, 1\}^3$.
    \begin{center}
    \begin{tikzcd}[sep=large]
    & \Fb \arrow[rr, "\begin{pmatrix}0 \\ 1\end{pmatrix}"] \arrow[from=dd]              &                           & \Fb^2            \\
    0 \arrow[rr, crossing over] \arrow[ru]                                               &                           & \Fb \arrow[ru, "\begin{pmatrix}1 \\ 0\end{pmatrix}" description]              &              \\
    & 0 \arrow[rr] &                           & \Fb \arrow[uu, "\begin{pmatrix}1 \\ 1\end{pmatrix}"'] \\
    0 \arrow[uu] \arrow[rr] \arrow[ru] &                           & 0 \arrow[uu, crossing over] \arrow[ru] &             
    \end{tikzcd}
    \end{center}
    This multipersistence module is degree 1, as every bicartesian square gets sent to a pullback of vector spaces. Furthermore, the multipersistence module is cross-degree 0, as every morphism in the diagram is a monomorphism.
    However, it is not projective (i.e., its projective dimension is not 0), as it is indecomposable and thus not free.
    The reason \autoref{thm:pdim_thm2} does not apply here is because the poset $\{0,1\}^3$ has dimension 3, not 2.
\end{example}

\begin{remark}
    We remark that everything in this section can be formally dualized. In particular, both \autoref{thm:pdim_thm1} and \autoref{thm:pdim_thm2} still hold if one exchanges projective resolutions with injective resolutions and calculus with cocalculus.
\end{remark}

\subsection{Examples}

We give several examples of the cross-(co)degree of multipersistence modules. 
In general, an $n$-parameter multipersistence module has cross-degree at most $n$, and this bound is sharp (and likewise for cross-codegree).
However, as we demonstrate in this section, several of the multifiltrations that are used in real-world applications give rise to multipersistence modules with low cross-degree or low cross-codegree.

\begin{example}
In \autoref{tab:degrees}, we list all the interval modules over $\{0,1\}^2$ and give their degree, cross-degree, codegree and cross-codegree.

\begin{table}[hbtp]
    \centering
    \begin{tabular}{c|c|c|c|c} % Three centered columns (c)
         & \; Degree \; & Cross-degree & Codegree & Cross-codegree \\
        {\begin{tikzcd}
            0 \arrow[r]            & 0           \\
            \Fb \arrow[u] \arrow[r] & 0 \arrow[u]
        \end{tikzcd}}
        & 2 & 2 & 1 & 0 \\
        {\begin{tikzcd}
            0 \arrow[r]            & 0           \\
            0 \arrow[u] \arrow[r] & \Fb \arrow[u]
        \end{tikzcd}},
        %& 2 & 1 & 2 & 1 \\
        {\begin{tikzcd}
            \Fb \arrow[r]            & 0           \\
            0 \arrow[u] \arrow[r] & 0 \arrow[u]
        \end{tikzcd}}
        & 2 & 1 & 2 & 1 \\
        {\begin{tikzcd}
            0 \arrow[r]            & \Fb           \\
            0 \arrow[u] \arrow[r] & 0 \arrow[u]
        \end{tikzcd}}
        & 1 & 0 & 2 & 2 \\
        {\begin{tikzcd}
            0 \arrow[r]            & 0           \\
            \Fb \arrow[u] \arrow[r] & \Fb \arrow[u]
        \end{tikzcd}},
        %& 1 & 1 & 1 & 0 \\
        {\begin{tikzcd}
            \Fb \arrow[r]            & 0           \\
            \Fb \arrow[u] \arrow[r] & 0 \arrow[u]
        \end{tikzcd}}
        & 1 & 1 & 1 & 0 \\
        {\begin{tikzcd}
            \Fb \arrow[r]            & \Fb           \\
            0 \arrow[u] \arrow[r] & 0 \arrow[u]
        \end{tikzcd}},
        {\begin{tikzcd}
            0 \arrow[r]            & \Fb           \\
            0 \arrow[u] \arrow[r] & \Fb \arrow[u]
        \end{tikzcd}}
        & 1 & 0 & 1 & 1 \\
        {\begin{tikzcd}
            \Fb \arrow[r]            & 0           \\
            \Fb \arrow[u] \arrow[r] & \Fb \arrow[u]
        \end{tikzcd}}
        & 2 & 1 & 2 & 0 \\
        {\begin{tikzcd}
            \Fb \arrow[r]            & \Fb           \\
            0 \arrow[u] \arrow[r] & \Fb \arrow[u]
        \end{tikzcd}}
        & 2 & 0 & 2 & 1 \\
        {\begin{tikzcd}
            \Fb \arrow[r]            & \Fb           \\
            \Fb \arrow[u] \arrow[r] & \Fb \arrow[u]
        \end{tikzcd}}
        & 0 & 0 & 0 & 0 \\
    \end{tabular}
    \caption{The possible interval modules $\{0,1\}^2 \to \Vecs_{\Fb}$, and their (co)degree and cross-(co)degree.}
    \label{tab:degrees}
\end{table}
    
\end{example}

\begin{example}
Given a finite metric space $X$ equipped with a function $f \colon X \to \R$, the \emph{sublevel-Rips} bifiltration (\cite[Definition 5.1]{BotnanLesnickMultipersistence}) is defined as the bifiltration
\begin{align*}
    F \colon &\R^2 \to \SCplx \\
    &(a,r) \mapsto \big\{\{x_1, \dots, x_n\} \subseteq X : f(x_i) \le a \ \forall i, d(x_i,x_j) \le r \ \forall i,j\big\}.
\end{align*}
A common choice of function in the sublevel-Rips filtration is the inverse of a density function, i.e., a function whose value is high in dense regions of the data (\cite[Example 5.2]{BotnanLesnickMultipersistence}, \cite[Example 2]{multipersistence_source_2}).

Let $F$ be the sublevel-Rips bifiltration of some metric space $X$ and $f \colon X \to \R$ a function. Then $H_0 F$ is a cross-codegree 1 bipersistence module.
To see this, observe that $H_0 F(a, r \le r')$ is an epimorphism for any $a$ and $r \le r'$ (as $F(a,r)_0 = F(a,r')_0$). Thus, in any square
\begin{equation*}
    \begin{tikzcd}
    {F(a,r')} \arrow[r]          & {F(a',r')}          \\
    {F(a,r)} \arrow[u] \arrow[r] & {F(a',r)} \arrow[u],
    \end{tikzcd}
\end{equation*}
the vertical arrows are epimorphisms, and hence the total cofiber is 0.

In \cite{elder_staircodes}, the persistence modules arising from $H_0$ of sublevel-Rips bifiltrations is studied. The authors show that due to the special structure of these persistence modules, one can define an efficiently computable barcode-like invariant.
\end{example}

We will now go on to describe a class of multipersistence modules whose cross-degree is 1. We first need the following lemma, which is a standard fact in algebraic topology.

\begin{lemma}\label{lemma:Hd_Rd}
    Let $X \subseteq \R^d$ be a finite CW complex embedded in $\R^d$, and let $R$ be a ring. For all $i \ge d$,
    \begin{equation*}
        H_i(X;R) \cong 0.
    \end{equation*}
\end{lemma}
\begin{proof}
    By \cite[Proposition 3.29]{hatcher}, open subsets $U$ of $\R^d$ satisfy
    \begin{equation*}
        H_i(U; R) \cong 0, \quad \textrm{for all } i \ge d.
    \end{equation*}
    By \cite[Theorem A.7]{hatcher}, every locally contractible compact subset of $\R^d$ is a retract of an open neighborhood. By \cite[Proposition A.4]{hatcher}, CW complexes are locally contractible. Furthermore, finite CW complexes are compact.
    Finally, an inclusion of a retract induce an injection on homology, which concludes the proof.
\end{proof}

Let $P$ be a poset. Given a topological space $W$ and function $f \colon W \to P$, the \emph{sublevel filtration} of $f$ is the functor
\begin{equation*}
    P \to \Top, \quad x \mapsto \{w \in W : f(w) \le x\}.
\end{equation*}
A filtration $F \colon P \to \Top$ is said to be \emph{1-critical} if it is isomorphic to the sublevel filtration of some function $f \colon \colim F \to P$ \cite{carlsson2009computing, blumberg2022universality}.

\begin{proposition}\label{prop:top_dim_homology_crossdeg1}
    Let $P$ be a lattice, and $F \colon P \to \Top$ a filtration of finite CW complexes in $\R^d$. If $F$ is 1-critical, then $H_{d-1} (F(-); R) \colon P \to \Mod_R$ is cross-degree 1.
\end{proposition}
\begin{proof}
    By \autoref{lemma:tofib_prod}, it suffices to show that for all $x, y \in P$, the morphism $H_{d-1} F(x \wedge y) \to H_{d-1} F(x) \oplus H_{d-1} F(y)$ is a monomorphism.
    
    As $F$ is 1-critical, there exists a function $f \colon \colim F \to P$ such that for all $x \in P$,
    \begin{equation*}
        F(x) = \{a \in \colim F : f(a) \le x\}.
    \end{equation*}
    Observe that for all $x, y \in P$,
    \begin{align*}
        &a \in F(x \wedge y) \iff f(a) \le x \wedge y 
        \iff f(a) \le x \textrm{ and } f(a) \le y \\
        \iff &a \in F(x) \textrm{ and } a \in F(y) \iff a \in F(x) \cap F(y).
    \end{align*}
    Hence, $F(x \wedge y) = F(x) \cap F(y)$.
    Now, consider the Mayer-Vietoris sequence associated to $F(x)$ and $F(y)$:
    \begin{equation*}
        \dots \to H_d(F(x) \cup F(y)) \to H_{d-1} F(x \wedge y) \to H_{d-1} F(x) \oplus H_{d-1} F(y) \to \dots.
    \end{equation*}
    By \autoref{lemma:Hd_Rd}, $H_d(F(x) \cup F(y)) \cong 0$. Hence, $H_{d-1}F(x \wedge y) \to H_{d-1} F(x) \oplus H_{d-1} F(y)$ is a monomorphism. This concludes the proof.
\end{proof}

Examples of finite CW complexes include finite simplicial complexes and finite cubical complexes.

\begin{example}\label{ex:colors}
    One common application of TDA is to analyze images (2D or 3D images, typically) using persistent homology. % TODO: add citation
    One models the image grid by a cubical complex and defines a real-valued function on the complex representing the pixel values, and then applies persistent homology to the sublevel filtration of this function.
    
    Suppose now that we have more than one color channel in the image (for example, red, green and blue), and we want to take those into account. One way to solve this is to construct a multifiltration. The different colors define functions $f_1, \dots, f_n$ from a cubical complex $X$ to $[m]$, where $[m] = \{0, \dots, m\}$ is the possible pixel values. These assemble into a single function
    \begin{equation*}
        f_1 \times \dots \times f_n \colon X \to [m]^n,
    \end{equation*}
    whose sublevel filtration gives a 1-critical multifiltration $F \colon [m]^n \to \Top$. This is a multifiltration of finite cubical complexes in $\R^d$, where $d$ is the dimension of the image.
    Hence,  the conditions of \autoref{prop:top_dim_homology_crossdeg1} are satisfied. 
    In particular, if the image is 2D, then $H_1 F$ is cross-degree 1, and if the image is 3D, then $H_2 F$ is cross-degree 1.
    
    This multifiltration is explained in more detail in \cite[Appendix C]{cumperlay}.
    An example of an application of multifiltrations induced by multi-channel images include \cite{carriere2020multiparameter}, where a bifiltration is constructed from a two-channel image. Instead of the channels representing different colors, one of the channels represent ``number of immune cells'', and the other ``number of cancer cells''.
\end{example}

\begin{remark}
    We further explain the implications of \autoref{ex:colors}.
    By \autoref{thm:pdim_thm1}, knowing the cross-degree tells us something about the projective dimension of our multipersistence module (i.e., the length of its minimal projective resolution).

    As an example, consider the case with two colors on a 2D image. Then, as explained in \autoref{ex:colors}, this gives rise to a bifiltration $F \colon [m]^2 \to \Top$, and $H_1 F$ has cross-degree 1.
    Thus, the bipersistence module $H_1 F \colon [m]^2 \to \Vecs_{\Fb}$ has projective at most dimension 1. On the other hand, arbitrary bipersistence modules can have projective dimension 2 (for example, a rectangle interval module has projective dimension 2).
\end{remark}

\bibliographystyle{plain}
\bibliography{ref}

\end{document}